\documentclass[review]{elsarticle}

\usepackage{lineno,hyperref}
\modulolinenumbers[5]
\journal{Name of the Journal}

\usepackage{amsmath,amssymb,amsthm,amsfonts,afterpage,float,bm,stmaryrd,paralist,wrapfig}

\usepackage{cancel}
\usepackage{color}
\usepackage{fancyhdr}
\usepackage{graphics}
\usepackage{hyperref}
\usepackage{graphicx,epsf}
\usepackage{makeidx}
\usepackage{epsfig}
\usepackage{lscape}

\pssilent

\newtheorem{theorem}{Theorem}[section]
\newtheorem{lemma}{Lemma}[section]
\newtheorem{proposition}{Proposition}[section]

\newtheorem{definition}{Definition}[section]
\newtheorem{remark}{Remark}[section]

\numberwithin{equation}{section}
\numberwithin{figure}{section}
\numberwithin{table}{section}

\def\XXint#1#2#3{{\setbox0=\hbox{$#1{#2#3}{\int}$}
\vcenter{\hbox{$#2#3$}}\kern-.51\wd0}}

\def\non{\nonumber }

\newcommand\defeq{\stackrel{\scriptscriptstyle \text{def}}=}
\newcommand\eps{{\epsilon}}
\newcommand{\RR}{\mathbb{R}}

\newcommand{\T}{\mathbb{T}^3}

\newcommand{\ud}{\mathrm{d}}

\graphicspath{{Figure/}}

\begin{document}

\setlength{\pdfpageheight}{\paperheight}
\setlength{\pdfpagewidth}{\paperwidth}
\title{Analysis for Allen-Cahn-Ohta-Nakazawa Model in a Ternary System}
\author{Sookyung Joo}
\address{Department of Mathematics and Statistics, Old Dominion University, Norfolk, VA, 23529}
\author{Xiang Xu}
\address{Department of Mathematics and Statistics, Old Dominion University, Norfolk, VA, 23529}
\author{Yanxiang Zhao}
\address{Department of Mathematics, George Washington University, Washington D.C., 20052}

\begin{abstract}
In this paper we study the global well-posedness of the Allen-Cahn Ohta-Nakazawa model with two fixed nonlinear volume constraints. Utilizing the gradient flow structure of its free energy, we prove the existence and uniqueness of the solution by following De Giorgi's minimizing movement scheme in a novel way.
\end{abstract}


\date{\today}
\maketitle


\section{Introduction}\label{sec:Introduction}

Ohta-Nakazawa (ON) model was originally introduced in \cite{OhtaNakazawa_Macromolecules1993} and has drawn much attention in materials science, particularly for the study of phase separation of triblock copolymers. Due to their remarkable ability for self-assembly into nanoscale ordered structures \cite{Hamley_Wiley2004}, triblock copolymers have generated much interest in materials engineering. Triblock copolymers are chain molecules made by three different segment species, say $A$, $B$ and $C$ species. Due to the chemical incompatibility, the three species tend to be phase-separated; on the other hand, the two species are connected by covalent chemical bonds, which leads to the so-called microphase separation. The ON model can describe such microphase separation for triblock copolymers by the ON free energy functional:
\begin{align}\label{OK-energy-original}
E^{\text{ON}}(\phi_1,\phi_2)&=\int_{\T}\Big[\dfrac{\eps}{2}\big(|\nabla\phi_1|^2+|\nabla\phi_2|^2+\nabla\phi_1\cdot\nabla\phi_2\big)
+\dfrac{1}{2\eps}W_{\mathrm{T}}(\phi_1,\phi_2)\Big]\,\ud{x} \nonumber\\
&\qquad+\displaystyle\sum_{i,j=1}^2\dfrac{\gamma_{ij}}{2}\int_{\T}\Big[(-\Delta)^{-\frac12}\big(f(\phi_i)-\omega_i\big)
\times(-\Delta)^{-\frac12}\big(f(\phi_j)-\omega_j\big)\Big]\,\ud{x}
\end{align}
Here $\T = \prod_{i=1}^3 [-X_i, X_i] \subset\RR^3$ denotes a periodic box, and $0<\eps\ll 1$ is an interface parameter that indicates the system is in the deep segregation regime. Phase field label functions that represent the density of $A$ and $B$ species are denoted by $\phi_i=\phi_i(x)$, $i=1, 2$, respectively. Meanwhile, the concentration of $C$ species can be implicitly represented by $1-\phi_1(x)-\phi_2(x)$ since the system is considered to be incompressible. The triple-well potential $W_{\mathrm{T}}$ is of the form
\[
W_{\mathrm{T}}(\phi_1, \phi_2):= W(\phi_1)+W(\phi_2)+W(1-\phi_1-\phi_2),
\]
with $W(s)=18(s^2-s)^2$. It is noted that $W_{\mathrm{T}}$ has three minima at $(1, 0, 0)$,
$(0, 1, 0)$ and $(0, 0, 1)$, which corresponds to the phase separation between the $A, B, C$ species.
It is also worth pointing out that the first
integral in \eqref{OK-energy-original} represents short-range interaction accounting for the interfacial free energy of the
system and favors large domains with small surface area, while the second integral term  in \eqref{OK-energy-original} describes long range interaction between chain molecules. We denote $\gamma_{ij}$, $i,j = 1, 2$ the strength of such long range interactions, and the constant matrix $[\gamma_{ij}]_{2\times 2}$ is assumed to be symmetric and positive definite.

The newly introduced term
\begin{align}
f(\phi_i) = 3\phi_i^2 - 2\phi_i^3, \qquad i = 1, 2
\end{align}
is adopted to mimic $\phi_i$, as the indicator for the $A$ and $B$ species, respectively. In our earlier work \cite{XuZhao_JSC2019, XuZhao_submitted2020}, a similar term has been introduced to some binary system with long-range interaction in order to study the associated $L^2$ gradient flow dynamics and maintain a better hyperbolic tangent profile for the solution and preserve its maximum principle at both continuous and discrete level. Meanwhile, we impose as usual fixed volume constraints
\begin{equation}\label{eqn:VolumeNoTime}
\overline{f(\phi_i)} := \frac{1}{|\T|}\int_{\T}f(\phi_i(x))\,\ud{x}=\omega_i, \qquad  i=1,2.
\end{equation}
For technical reasons (see the proof of Proposition \ref{proposition-convergence}), we assume
\begin{align}\label{tech-assumption}
  \omega_i\neq 0, 1, \qquad i = 1, 2,
\end{align}
namely, no any single species occupies the entire region of $\T$. The operator  $(-\Delta)^{-\frac12}u$ is the square root of the operator $(-\Delta)^{-1} u$ with periodic boundary condition. Note that $u$ has to be zero mean to well define the operator $(-\Delta)^{-1} u$, we will take $(-\Delta)^{-1} u : = (-\Delta)^{-1} (u-\overline{u})$ when $u$ is not zero mean. In other words, removal of the zeroth Fourier mode for $u$ will make $(-\Delta)^{-1} u$ always well defined. Besides, hereafter for any function $u$, we always set $(-\Delta)^{-1} u$ and $(-\Delta)^{-\frac12}u$ to be with zero mean.

In order to study the equilibria of the ON model, we consider the $L^2$ gradient flow $\partial_t\phi_i=-\delta{E}/\delta\phi_i-\lambda_i{f}'(\phi_i)$ generated by the ON energy functional
\begin{align}
&\partial_t\phi_i= \epsilon\Delta\phi_i+\frac{\epsilon\Delta\phi_j}{2}-\frac{1}{2\epsilon}\frac{\partial{W}_{\mathrm{T}}}{\partial\phi_i}
-\sum_{k=1}^2\gamma_{ik}(-\Delta)^{-1}\big(f(\phi_k)-\omega_k\big)f'(\phi_i)-\lambda_i(t)f'(\phi_i), \label{eqn:ACON}\\
&\phi_i(x,0)=\phi_{i0}(x),\label{bc:ACON}
\end{align}
for $ (x,t)\in \T\times(0,+\infty)$, $i, j = 1, 2$ and $i\neq j$, subject to the volume constraints
\begin{equation}\label{eqn:Volume}
\overline{f(\phi_i(t))}=\dfrac{1}{|\T|}\int_{\T}f(\phi_i(t, x))\,\ud{x}=\omega_i,
\quad\forall t\in[0, +\infty),\;\omega_i\neq 0, 1.
\end{equation}
Here $\omega_i$ are given constants and $\lambda_i(t)$ are the
corresponding Lagrange multipliers to \eqref{eqn:Volume}:
\begin{equation}\label{lambda}
\lambda_i=\dfrac{\int_{\T}-\frac{\delta{E}^{\text{ON}}}{\delta\phi_i}f'(\phi_i)\,\ud{x}}{\int_{\T}|f'(\phi_i)|^2\,\ud{x}}.
\end{equation}
Hereafter, we will refer (\ref{eqn:ACON})-(\ref{bc:ACON}) to as Allen-Cahn-Ohta-Nakazawa(ACON) equations. If $(\phi_1(x,t), \phi_2(x,t))$ is a solution of the ACON dynamics (\ref{eqn:ACON})-(\ref{bc:ACON}), it is well known that it satisfies:
\begin{align}\label{eqn:EnergyDissipation}
\frac{\text{d}}{\text{d}t} E^{\text{ON}}(\phi_1,\phi_2) = - \int_{\T} \left|\frac{\delta E^{\text{ON}}}{\delta\phi_1}\right|^2 +  \left|\frac{\delta E^{\text{ON}}}{\delta\phi_2}\right|^2\ \text{d}x \le 0,
\end{align}
which implies that the ON energy is decreasing along the solution trajectory $(\phi_1(x,t), \phi_2(x,t))$. This is the so-called energy dissipation law for the general gradient flow dynamics.

The contribution in this work is that we prove the existence and uniqueness of the solution for ACON system by following the De Giorgi's
minimizing movement scheme in a novel way. Different from all existing literature in utilizing this classical implicit Euler scheme to derive the Euler Lagrange equations at the discrete level, we identify the limit curve first and use an approximation of this limit curve to establish the nonlinear terms caused by the nonlinear volume constraints in discrete Euler Lagrange equations. 

De Giorgi's minimizing movement scheme \cite{Ambrosio_Rend1995,DeGiorgi_RMA1993}, which is also referred to as the Rothe's method, is an implicit Euler scheme specialized at gradient flows in separable Hilbert spaces (later extended to general metric spaces). Given a gradient flow $\partial_tu=-\nabla{F}(u)$, where the energy $F$ is coercive and lower semicontinuous, this very scheme provides an energy-driven implicit-time discretization to solve the evolution equation within a natural framework. Considering the gradient flow \eqref{GF} with the nonlinear volume constraints \eqref{eqn:Volume}, the great advantage to apply De Giorgi's minimizing movement scheme is that it ensures the preservation of such volume constraints at each discrete step. Nevertheless, to prove the existence of solutions to \eqref{GF} there are still essential difficulties arising from \eqref{eqn:Volume}: after acquiring a discrete sequence $(\phi_{1\tau}^k, \phi_{2\tau}^k)$, usually the next step is to establish the Euler Lagrange equations for this discrete sequence; however, if we follow the standard procedures the denominators in the corresponding Lagrange multiplier
terms caused by \eqref{eqn:Volume} cannot be ensured to be nonzero. To solve this issue, alternatively we identify the limit curve $(\phi_1(t), \phi_2(t))$ of the piecewise constant interpolation functional $(\phi_{1\tau}(t), \phi_{2\tau}(t))$ as $\tau\rightarrow 0$ first,
based on the uniform bounds achieved in Lemma \ref{estimates-discrete}. The assumption \eqref{tech-assumption} together with the refined Arzela-Ascoli theorem in \cite{AmbrosioGigliSavare_Book2008} ensures that the quantities $\int_{\mathbb{T}^3}|f'(\phi_i(t))|^2dx$, $i=1, 2$, related to the limit curve, stay away from $0$, which plays the crucial role to further derive the Euler Lagrange equations of the discrete sequence as well as the uniform bound of the discrete Lagrange multipliers. To the best of our knowledge, this has been the first time that the De Giorgi's minimizing movement scheme
is used in such a manner. Meanwhile, we also want to point out that the success of such derivation might be undermined due to the non-integrability of certain terms in the discrete Lagrange multipliers. Therefore, instead of using the limit curve directly, we shall perform approximations first by virtue of the classical resolvant operator $J_\lambda=(I-\lambda\Delta)^{-1}$ for sufficiently small $\lambda>0$.

Some conventional notations adopted throughout the paper are collected here. We will denote by $\|\cdot\|_{L^p}$ and $\|\cdot\|_{H^s}$ the standard norms for the periodic Sobolev spaces $L^p_{\text{per}}(\T)$ and $H^s_{\text{per}}(\T)$. The standard $L^2$ inner product will be denoted by $\langle \cdot, \cdot \rangle$.


\section{Existence and Uniqueness of the Solution of ACON system}\label{section:Wellposedness}

Without loss of generality, throughout this section, we consider $\eps=1$ in \eqref{OK-energy-original} and $|\T| = 1$. Accordingly $E^{\text{ON}}$ is replaced by an energy functional $E$:
\begin{align}\label{OK-energy}
E(\phi_1,\phi_2)&=\int_{\T}\left[\dfrac{1}{2}\Big(|\nabla\phi_1|^2+|\nabla\phi_2|^2+\nabla\phi_1\cdot\nabla\phi_2\Big)
+\dfrac{1}{2}W_{\mathrm{T}}(\phi_1,\phi_2)\right]\,\ud{x} \nonumber\\
&\qquad+\displaystyle\sum_{i,j=1}^2\dfrac{\gamma_{ij}}{2}\int_{\T}\left[(-\Delta)^{-\frac12}\big(f(\phi_i)-\omega_i\big)
\times(-\Delta)^{-\frac12}\big(f(\phi_j)-\omega_j\big)\right]\,\ud{x}.
\end{align}
The associated $L^2$ gradient flow dynamics (\ref{eqn:ACON})-(\ref{bc:ACON}) is replaced by
\begin{align}
&\partial_t\phi_i=\Delta\phi_i+\frac{\Delta\phi_j}{2}-\frac12\frac{\partial{W}_{\mathrm{T}}}{\partial\phi_i}
-\sum_{k=1}^2\gamma_{ik}(-\Delta)^{-1}\big(f(\phi_k)-\omega_k\big)f'(\phi_i)-\lambda_i(t)f'(\phi_i), \label{GF}\\
&\phi_i(x,0)=\phi_{i0}(x),\label{IC}
\end{align}
for $ (x,t)\in \T\times(0,+\infty)$, $i, j = 1, 2$ and $i\neq j$, subject to the volume constraints \eqref{eqn:Volume}. Here $\lambda_i(t)$ are the corresponding Lagrange multipliers to \eqref{eqn:Volume} by replacing $E^{\text{ON}}$ by $E$:
\begin{equation}\label{lambda}
\lambda_i=\dfrac{\int_{\T}-\frac{\delta{E}}{\delta\phi_i}f'(\phi_i)\,\ud{x}}{\int_{\T}|f'(\phi_i)|^2\,\ud{x}}.
\end{equation}

\subsection{Implicit Euler Scheme}\label{subsection:ImplicitEulerScheme}

We define functional spaces
\begin{align}
H^{1}_{\omega_i}=\left\{ u \in H^1(\T), \overline{f(u)}=\omega_i\right\}, \quad i=1, 2,
\end{align}
and start the argument from the following lemma.

\begin{lemma}\label{lemma-H1-bound}
For any $\phi_i \in H_{\omega_i}^1, i=1,2$, one has
\begin{equation}
\|\phi_i\|_{H^1(\T)}^2\leq 4E(\phi_1,\phi_2)+2.
\end{equation}
\end{lemma}

\begin{proof}
Using Young's inequality, we get
\[
\phi_i^2\leq\dfrac{\phi_i^4}{4}+1=\dfrac{\phi_i^2(\phi_i-1+1)^2}{4}+1\leq
\dfrac{\phi_i^2[2(\phi_i-1)^2+2]}{4}+1=\dfrac{\phi_i^2(\phi_i-1)^2}{2}+\frac{\phi_i^2}{2}+1.
\]
Hence
\[
\int_{\T}|\phi_i(x)|^2\,\ud{x}\leq\int_{\T}(\phi_i^2-\phi_i)^2\ud{x}+2|\T|=\int_{\T}(\phi_i^2-\phi_i)^2\ud{x}+2 \le \int_{\T} W(\phi_1)\ud{x} + 2.
\]
Note that the energy $E(\phi_1, \phi_2)$ can be rewritten as
\begin{align*}
E(\phi_1,\phi_2)&=\dfrac{1}{2}\int_{\T}\left[\dfrac{1}{2}|\nabla\phi_1|^2+ W(\phi_1)\right]+
                                                            \left[\dfrac{1}{2}|\nabla\phi_2|^2+ W(\phi_2)\right]+
                                                            \left[\dfrac{1}{2}|\nabla(1-\phi_1-\phi_2)|^2+ W(1-\phi_1-\phi_2)\right] \ud{x} \\
&\qquad+\displaystyle\sum_{i,j=1}^2\dfrac{\gamma_{ij}}{2}\int_{\T}\left[(-\Delta)^{-\frac12}\big(f(\phi_i)-\omega_i\big)
\times(-\Delta)^{-\frac12}\big(f(\phi_j)-\omega_j\big)\right]\,\ud{x}
\end{align*}
and the positive definiteness of $[\gamma_{ij}]$ implies the nonnegativity of the second part of $E(\phi_1,\phi_2)$, the proof is finished by using \eqref{OK-energy}.
\end{proof}

Next, for any fixed time step $\tau>0$ and $(\phi_1^\ast, \phi_2^\ast)\in L^2(\T)\times L^2(\T)$,
we consider the functional
\begin{equation}\label{discrete-functional}
F_\tau(\phi_1,\phi_2; \phi_1^*,\phi_2^*) = E(\phi_1,\phi_2) + \dfrac{\|\phi_1-\phi_1^\ast\|_{L^2(\T)}^2+\|\phi_2-\phi_2^\ast\|_{L^2(\T)}^2}{2\tau},\quad
\phi_i\in H^1_{\omega_i}(\T),\; i=1, 2.
\end{equation}
and prove the existence of its minimizers. To this end, we need to derive the following inequalities first.

\begin{lemma}\label{lemma-key-inequality}
Let $w\in L^\frac{6}{5}(\T)$ and $\Psi=(-\Delta)^{-1}w=G\ast w$, where $G$ is the Green's function for the Laplacian operator coupled with periodic boundary condition.
Then there exists a generic constant $C>0$, such that
\begin{align}
\|\Psi\|_{L^6(\T)}&  \leq C\|w\|_{L^{\frac65}(\T)},\label{HLS-1}  \\
\|\nabla\Psi\|_{L^2(\T)}&\leq C\|w\|_{L^{\frac65}(\T)}.\label{HLS-2}
\end{align}
\end{lemma}

\begin{proof}
Since $\Psi$ satisfies
\[
  \begin{cases}
   -\Delta\Psi=w,\\
   \int_{\T}\Psi(x)\,dx=0,
  \end{cases}
\]
we multiply both sides of the first equation above by $\Psi$, and then integrate over $\T$. It yields
\[
  \|\nabla\Psi\|_{L^2(\T)}^2=\int_{\T}w(x)\Psi(x)\,dx\leq\|\Psi\|_{L^6(\T)}\|w\|_{L^{\frac65}(\T)}
  \leq C\|\Psi\|_{H^1(\T)}\|w\|_{L^\frac{6}{5}(\T)} \le C\|\nabla\Psi\|_{L^2(\T)}\|w\|_{L^\frac{6}{5}(\T)}
\]
in which the last inequality is due to Poincare's inequality and hence \eqref{HLS-2} is proved. Furthermore, $\eqref{HLS-1}$ comes directly from the Sobolev inequality $\|\Psi\|_{L^6(\T)}\leq C\|\Psi\|_{H^1(\T)}$, Poincare's inequality and \eqref{HLS-2}.
\end{proof}

By Lemmas \ref{lemma-H1-bound} and \ref{lemma-key-inequality}, it is immediate to check that
\begin{lemma}\label{lemma-discrete-minimizer}
The functional $F_\tau$ has a minimizer in $H_{\omega_1}^1\times H_{\omega_2}^1$.
\end{lemma}

\begin{proof}
First, $F_\tau$ is nonnegative thus there exists a minimizing
sequence $\{(\phi_{1n},\phi_{2n})\}$ in $H^1_{\omega_1}\times H^1_{\omega_2}$ satisfying
\[
0\leq \displaystyle\inf_{\psi_i\in H^1_{\omega_i}}F_\tau(\psi_1,\psi_2)\leq F_\tau(\phi_{1n},\phi_{2n})\leq \displaystyle
  \inf_{\psi_i\in H^1_{\omega_i}}F_\tau(\psi_1,\psi_2)+\dfrac{1}{n}.
\]
Hence $E(\phi_{1n},\phi_{2n})$ is bounded. By Lemma \ref{lemma-H1-bound}, $\{(\phi_{1n},\phi_{2n})\}$ is bounded in $H^1(\T)\times H^1(\T)$, and (up to a subsequence) we get $\phi_{in_k} \rightharpoonup \phi_i$ weakly in $H^1(\T)$, $\phi_{in_k}\rightarrow \phi_i$ strongly in $L^p(\T)$, for $p\in[1,6)$ and $i=1, 2$. Hence $\phi_i\in H^1_{\omega_i}$, and we can further derive
\begin{align*}
&\quad \int_{\T}\big(|\nabla\phi_1|^2+|\nabla\phi_2|^2+\nabla\phi_1\cdot\nabla\phi_2\big)\,\ud{x} \\
&= \int_{\T} \frac12 |\nabla\phi_1+\nabla\phi_2|^2\,\ud{x} +
  \int_{\T} \frac12 |\nabla\phi_1|^2\,\ud{x}+
   \int_{\T} \frac12 |\nabla\phi_2|^2\,\ud{x}\\
&\leq\displaystyle\liminf_{k\rightarrow\infty}\int_{\T} \frac12 |\nabla\phi_{1n_k}+\nabla\phi_{2n_k}|^2\,\ud{x}
+\displaystyle\liminf_{k\rightarrow\infty}\int_{\T} \frac12 |\nabla\phi_{1n_k}|^2\,\ud{x}
+\displaystyle\liminf_{k\rightarrow\infty}\int_{\T} \frac12 |\nabla\phi_{2n_k}|^2\,\ud{x}\\
&=\displaystyle\liminf_{k\rightarrow\infty}\int_{\T}\big(|\nabla\phi_{1n_k}|^2+|\nabla\phi_{2n_k}|^2
+\nabla\phi_{1n_k}\cdot\nabla\phi_{2n_k}\big)\,\ud{x},
\end{align*}
and
\[
\displaystyle\lim_{k\rightarrow\infty}W_{\mathrm{T}}(\phi_{1n_k},\phi_{2n_k})=W_{\mathrm{T}}(\phi_1,\phi_2).
\]
Besides, denoting
$\Psi_{in_k}=G\ast\big(f(\phi_{in_k})-\omega_i\big)$,
$\Psi_i=G\ast\big(f(\phi_i)-\omega_i\big)$, by \eqref{HLS-2}, it yields
$$
 0\leq\displaystyle\lim_{k\rightarrow\infty}\|\nabla\Psi_{in_k}-\nabla\Psi_i\|_{L^2(\T)}\leq
 C\displaystyle\lim_{k\rightarrow\infty}\|f(\phi_{in_k})-f(\phi_i)\|_{L^{\frac{6}{5}}(\T)}=0,
$$
which implies $\nabla\Psi_{in_k}\rightarrow\nabla\Psi_i$ in $L^2(\T)$.
As a consequence, we have
\begin{align*}
&\displaystyle\lim_{k\rightarrow\infty}\int_{\T}(-\Delta)^{-\frac12}\big(f(\phi_{in_k})-\omega_i\big)
\times(-\Delta)^{-\frac12}\big(f(\phi_{jn_k})-\omega_j\big)\,\ud{x} \\
=&\displaystyle\lim_{k\rightarrow\infty}\int_{\T}\nabla\Psi_{in_k}\cdot\nabla\Psi_{jn_k}\,\ud{x}
=\int_{\T}\nabla\Psi_{i}\cdot\nabla\Psi_{j}\,\ud{x}\\
=&\int_{\T}(-\Delta)^{-\frac12}\big(f(\phi_i)-\omega_i\big)\times(-\Delta)^{-\frac12}\big(f(\phi_j)-\omega_j\big)\,\ud{x}.
\end{align*}
To sum up, we conclude that
$$
  E(\phi_1,\phi_2)\leq \displaystyle\liminf_{k\rightarrow\infty}E(\phi_{1n_k},\phi_{2n_k}),
$$
and henceforth
$$
  F_\tau(\phi_1,\phi_2)\leq\displaystyle\liminf_{k\rightarrow\infty}F_\tau(\phi_{1n_k},\phi_{2n_k})
  = \displaystyle\inf_{\psi_i\in H^1_{\omega_i} }F_\tau(\psi_1,\psi_2),
$$
which finishes the proof.
\end{proof}

As a consequence, for any initial data $(\phi_{10}, \phi_{20})\in H^1_{\omega_1}\times H^1_{\omega_2}$, using Lemma \ref{lemma-discrete-minimizer} one may define a discrete sequence $\{(\phi_{1\tau}^k, \phi_{2\tau}^k)\}$ recursively by
\begin{equation}\label{def-discrete-sequence}
\begin{cases}
(\phi_{1\tau}^0,\phi_{2\tau}^0)=(\phi_{10},\phi_{20}),\\
(\phi_{1\tau}^{k+1},\phi_{2\tau}^{k+1}):=\underset{\phi_i\in H_{\omega_i}^1}{\mathrm{arg min}} \ F_{\tau}(\phi_1, \phi_2;\phi_{1\tau}^k, \phi_{2\tau}^k), \;\forall k\geq 0.
\end{cases}
\end{equation}
Correspondingly, we consider a piecewise constant interpolation
$t\in [0, +\infty)\mapsto (\phi_{1\tau}(t),\phi_{2\tau}(t))$ by
\begin{equation}\label{def-piecewise-interpolation}
\big(\phi_{1\tau}(t),\phi_{2\tau}(t)\big)=(\phi_{1\tau}^k,\phi_{2\tau}^k) \quad\mbox{for }\; k\tau\leq t<(k+1)\tau.
\end{equation}
Then we can collect the following estimates for the piecewise constant functional $(\phi_{1\tau}(t),\phi_{2\tau}(t))$.

\begin{lemma}\label{estimates-discrete}
For any $T>0$, $\tau\in (0, 1)$ and $0\leq s<t\leq T$, the piecewise constant interpolation functional $(\phi_{1\tau}(t),\phi_{2\tau}(t))$ satisfies that for $i=1,2$,
\begin{align}
\displaystyle\sup_{t\in [0, T]}\|\phi_{i\tau}(t)\|_{H^1(\T)}&\leq \sqrt{4E(\phi_{10},\phi_{20})+2},\label{uniform-H1-bound}\\
E\big(\phi_{1\tau}(t),\phi_{2\tau}(t)\big)&\leq E\big(\phi_{1\tau}(s),\phi_{2\tau}(s)\big)\leq E(\phi_{10},\phi_{20}), \label{energy-decreasing}\\
\|\phi_{i\tau}(t)-\phi_{i\tau}(s)\|_{L^p(\T)}&\leq C(\phi_{10},\phi_{20},p)(t-s+\tau)^{\frac{6-p}{4p}},\qquad\forall\, p\in[2,6).
\label{time-difference-estimate}
\end{align}
\end{lemma}

\begin{proof}

First, since $(\phi_{1\tau}^{k+1},\phi_{2\tau}^{k+1})$ is a minimizer of $F_\tau$ with
$(\phi_1^\ast, \phi_2^\ast)=(\phi_{1\tau}^k, \phi_{2\tau}^k)$, we know that
\begin{align}\label{discrete-inequality}
E(\phi_{1\tau}^{k+1},\phi_{2\tau}^{k+1}) +
\dfrac{1}{2\tau}\Big(\big\|\phi_{1\tau}^{k+1}-\phi_{1\tau}^{k}\big\|_{L^2(\T)}^2+\big\|\phi_{2\tau}^{k+1}-\phi_{2\tau}^{k}\big\|_{L^2(\T)}^2\Big)
\leq
E(\phi_{1\tau}^{k}, \phi_{2\tau}^{k}),\quad\forall k\geq 0,
\end{align}
which implies that for $\forall t\in [0, T]$ it holds
\begin{align*}
E\big(\phi_{1\tau}(t), \phi_{2\tau}(t)\big)\leq E(\phi_{1\tau}^0, \phi_{2\tau}^0)=E(\phi_{10},\phi_{20}).
\end{align*}
As a consequence, it follows from Lemma \ref{lemma-H1-bound} that
\begin{align*}
\big\|\phi_{i\tau}(t)\big\|_{H^1(\T)}^2\leq 4E\big(\phi_{1\tau}(t), \phi_{2\tau}(t)\big)+2\leq
4E(\phi_{10},\phi_{20})+2.
\end{align*}
Moreover, $\forall\, 0\leq s<t\leq T$, let us denote $m=\lfloor s/\tau\rfloor$, $n=\lfloor t/\tau\rfloor$. Repeated use of \eqref{discrete-inequality} directly yields
\begin{align*}
E\big(\phi_{1\tau}(t),\phi_{2\tau}(t)\big)=E(\phi_{1\tau}^n,\phi_{2\tau}^n)\leq
E(\phi_{1\tau}^m,\phi_{2\tau}^m)=E\big(\phi_{1\tau}(s),\phi_{2\tau}(s)\big).
\end{align*}
Meanwhile using H\"{o}lder's inequality and summing \eqref{discrete-inequality} over $k=m,\cdots, n-1$ we obtain
\begin{align*}
\big\|\phi_{i\tau}^n-\phi_{i\tau}^m\big\|_{L^2(\T)}&\leq\displaystyle\sum_{k=m}^{n-1}\big\|\phi_{i\tau}^{k+1}-\phi_{i\tau}^{k}\big\|_{L^2(\T)}
\leq\sqrt{n-m}\sqrt{\displaystyle\sum_{k=m}^{n-1}\big\|\phi_{i\tau}^{k+1}-\phi_{i\tau}^{k}\big\|_{L^2(\T)}^2}\\
&\leq\sqrt{2\tau{n}-2\tau{m}}\sqrt{E(\phi_{10},\phi_{20})}, \quad i=1,2
\end{align*}
which further indicates
\begin{equation}\label{bound-L2-difference}
\big\|\phi_{i\tau}(t)-\phi_{i\tau}(s)\big\|_{L^2(\T)}\leq
\sqrt{2E(\phi_{10},\phi_{20})}\sqrt{t-s+\tau}.
\end{equation}
Therefore, using Sobolev interpolation, \eqref{uniform-H1-bound} and \eqref{bound-L2-difference}, we have that for $\forall\, p\in[2,6)$
\begin{align*}
\|\phi_{i\tau}(t)-\phi_{i\tau}(s)\|_{L^p(\T)}
&\leq C\|\phi_{i\tau}(t)-\phi_{i\tau}(s)\|_{L^2(\T)}^{\frac{6-p}{2p}}\|\phi_{i\tau}(t)-\phi_{i\tau}(s)\|_{L^6(\T)}^{\frac{3p-6}{2p}}\non\\
&\leq C\|\phi_{i\tau}(t)-\phi_{i\tau}(s)\|_{L^2(\T)}^{\frac{6-p}{2p}}\|\phi_{i\tau}(t)-\phi_{i\tau}(s)\|_{H^1(\T)}^{\frac{3p-6}{2p}}\non\\
&\leq C(t-s+\tau)^{\frac{6-p}{4p}} (4E(\phi_{10},\phi_{20})+2)^{\frac{3p-6}{2p}}
\end{align*}
which leads to \eqref{time-difference-estimate}.
\end{proof}

It immediately follows from Lemma \ref{estimates-discrete} that
\begin{proposition}\label{proposition-convergence}
There exists a sequence $\{\tau_n\}\searrow 0^+$, such that
\begin{align}\label{Lp-strong-convergence}
\begin{cases}
\phi_{1\tau_n}(t)\rightarrow\phi_1(t)\\
\phi_{2\tau_n}(t)\rightarrow\phi_2(t)
\end{cases}
\;\;\mbox{strongly in } L^p(\T),
\;\;\forall\,t\in [0, T],\ \forall p\in[2,6),
\end{align}
where $\phi_1, \phi_2\in C([0, T]; L^p(\T))\cap L^\infty(0, T; H^1(\T))$. Besides, for $\forall\,t\in [0, T]$ and $i=1,2$
\begin{align}
&\overline{f(\phi_i(t))}=\frac{1}{|\T|}\int_{\T}f(\phi_i(t,x))\,\ud{x}=\omega_i, \label{trace-zero}\\
&\int_{\T}\big|f'(\phi_i(t,x))\big|^2\,\ud{x}=36\int_{\T} (\phi_i^2(t,x)-\phi_i(t,x))^2\,\ud{x}\geq\beta, \label{positivity-1}
\end{align}
where $\beta>0$ is a generic constant.

\end{proposition}

\begin{proof}
To begin with, using \eqref{uniform-H1-bound}, \eqref{time-difference-estimate}, the compact embedding of $H^1(\T)$ into $L^p(\T)$, and a celebrated refined version of the Ascoli-Arzela theorem (see \cite[Proposition 3.3.1]{AmbrosioGigliSavare_Book2008}), we can extract a subsequence $\tau_n\searrow 0^+$, such that
\begin{align}\label{L2-strong-convergence}
\begin{cases}
\phi_{1\tau_n}(t)\rightarrow\phi_1(t)\\
\phi_{2\tau_n}(t)\rightarrow\phi_2(t)
\end{cases} \;\;\mbox{strongly in } L^p(\T),
\;\forall\, t\in [0, T],\ \forall p\in[2,6),
\end{align}
and
\begin{equation}\label{continuity-L2}
\phi_i\in C([0, T]; L^p(\T)), \quad i=1,2.
\end{equation}
Moreover, it is easy to check \eqref{trace-zero} is also valid.

To prove \eqref{positivity-1}, suppose there exists $\tilde{t}\in[0, T]$, such that
\[
f'(\phi_i(\tilde{t}, x))= 6( \phi_i(\tilde{t}, x)-\phi_i^2(\tilde{t}, x) )=0 \quad\mbox{a.e. in
}\,\T.
\]
That is, $\phi_i(\tilde{t}, x)=0$ or $1$ a.e. in $\T$. Approximating $f'$ by bounded functions, we deduce $\nabla f(\phi_i(\tilde{t}, x))=f'(\phi_i(\tilde{t}, x))\nabla\phi_i=0$ a.e. in $\T$. Hence $f(\phi_i(\tilde{t}))$ is constant a.e. in $\T$. Since $f(0)\neq f(1)$, so either $\phi_i(\tilde{t}, x)=0$ a.e. in $\T$ or $\phi_i(\tilde{t}, x)=1$ a.e. in $\T$. But neither case results in \eqref{trace-zero} because  $\omega_i\neq 0, 1$ in \eqref{eqn:Volume}. Therefore, $\int_{\T}\big|f'(\phi_i(t,x))\big|^2dx>0$, $\forall t\in [0, T]$. Thus \eqref{positivity-1} is valid due to the fact that $\phi\in C([0, T]; L^p(\T))$, for $2\leq p<6$.
\end{proof}

\subsection{Euler-Lagrange Equations for the Discrete Sequence}

Before the derivation of the Euler-Lagrange equations for the discrete sequence $\big\{(\phi_{1\tau_n}^k, \phi_{2\tau_n}^k)\big\}$,
we need to first establish from \eqref{positivity-1} the following result concerning the approximation of
the limit curve $(\phi_1, \phi_2)$ by more regular functions. Such approximation is necessary, otherwise some terms in the Lagrange multipliers
could not be kept under control (see Remark \eqref{remark-approximation} below for details)

\begin{proposition}\label{proposition-positivity-2}
Let $(\phi_1, \phi_2)$ be the limits in Proposition \ref{proposition-convergence}, then there exists
$\xi_1, \xi_2\in C([0, T]; W^{2,p}(\T))$, $\forall\, p \in[2,6)$, satisfying
\begin{equation}
\int_{\T}\big|f'(\xi_i(t, x))-f'(\phi_i(t,
x))\big|^2\,\ud{x}\leq\frac{\beta}{16}, \quad\forall\, t\in [0, T],\, i=1,2.
\end{equation}

\end{proposition}

\begin{proof}
It suffices to prove for $i=1$. First,
it is easy to check $-\Delta: W^{2, p}(\T)\rightarrow L^p(\T)$
is an infinitesimal generator of a linear semigroup of contractions. For any $\lambda>0$, we consider the resolvent operator
$J_\lambda=(I-\lambda\Delta)^{-1}$. Then $J_\lambda$ is a linear bounded operator from $L^p(\T)$ into itself, and (see \cite[Lemma 2.2.1]{Zheng_Book2004})
\begin{equation}\label{resolvent-bounded}
  \|J_\lambda\|\leq 1, \quad\forall\lambda>0.
\end{equation}
Since $\phi_1\in C([0, T]; L^p(\T))$, $\forall \, \varepsilon>0$, there exists $\tilde\delta=\tilde\delta(\varepsilon)>0$ such that
\begin{equation}\label{small-delta-0}
  \|\phi_1(s)-\phi_1(\tilde{s})\|_{L^p(\T)}<\frac{\varepsilon}{3}, \quad\mbox{whenever }\; |s-\tilde{s}|<\tilde\delta.
\end{equation}
Choosing $K\in\mathbb{N}$ sufficiently big such that $T/K<\tilde\delta$, and letting
$$
  t_m=\frac{mT}{K}, \quad\forall\, 0\leq m\leq K
$$
By \cite[Lemma 2.2.1]{Zheng_Book2004}, there exists $\delta=\delta(\varepsilon)>0$ such that
\begin{equation}\label{small-delta-1}
  \|J_\lambda\phi_1(t_m)-\phi_1(t_m)\|_{L^p(\T)}\leq\frac{\varepsilon}{3},\quad\forall\, 0<\lambda<\delta,\; 0\leq m\leq K.
\end{equation}

In all, for any $t\in [0, T]$, there exists some $t_j$ ($0\leq j\leq K$), such that $|t-t_j|<\tilde\delta$, hence
we get from \eqref{resolvent-bounded}, \eqref{small-delta-0} that
\begin{align*}
\|J_\lambda\phi_1(t)-\phi_1(t)\|_{L^p(\T)}
&\leq\|J_\lambda\phi_1(t)-J_\lambda\phi_1(t_j)\|_{L^p(\T)}+\|J_\lambda\phi_1(t_j)-\phi_1(t_j)\|_{L^p(\T)}
+\|\phi_1(t_j)-\phi_1(t)\|_{L^p(\T)}\nonumber\\
&\leq \|\phi_1(t)-\phi_1(t_j)\|_{L^p(\T)}+\|J_\lambda\phi_1(t_j)-\phi_1(t_j)\|_{L^p(\T)}+\|\phi_1(t_j)-\phi_1(t)\|_{L^p(\T)}\nonumber\\
&\leq \frac{\varepsilon}{3}+\frac{\varepsilon}{3}+\frac{\varepsilon}{3} \leq \varepsilon,
\end{align*}
provided $\lambda<\delta$. Note that $\phi_1(t)\in C([0, T]; L^p(\T))$ and \eqref{resolvent-bounded} implies
$J_\lambda\phi_1(t)\in C([0, T]; W^{2,p}(\T))$.

\smallskip

Finally, choosing $\lambda$ sufficiently small and setting $\xi_1=J_\lambda\phi_1$, we finish the proof.
\end{proof}

To proceed our proof, we shall show that along the decreasing sequence $\{\tau_n\}$, the minimizers to $F_{\tau_n}$ satisfy the following Euler-Lagrange equations provided $n$ is sufficiently large. To simplify the notation, we denote $N_n = \lfloor T/\tau_n \rfloor$.

\begin{remark}
It is worth mentioning that we only consider the rest of the sequence $\{\tau_n\}$ as $n$ becomes large enough
because it ensures the denominator in the lagrange multipliers will be kept away from zero, see \eqref{positivity-main} below.
\end{remark}

\begin{lemma}\label{lemma-discrete-EL}
There exists $N=N(\beta)\in\mathbb{N}$, such that $\forall n\geq N$,
we have
\begin{align}\label{equation-discrete-EL}
&\int_{\T}\Big[\dfrac{\phi_{i\tau_n}^{k+1}-\phi_{i\tau_n}^k}{\tau_n}-\frac12\Delta\phi_{1\tau_n}^{k+1}-\frac12\Delta\phi_{2\tau_n}^{k+1}
-\frac12\Delta\phi_{i\tau_n}^{k+1}+ \dfrac{1}{2}W'(\phi_{i\tau_n}^{k+1}) - \dfrac{1}{2} W'(1-\phi_{1\tau_n}^{k+1}-\phi_{2\tau_n}^{k+1}) \non\\
&\qquad+\sum_{j=1}^2\gamma_{ij}(-\Delta)^{-1}\big(f(\phi_{j\tau_n}^{k+1})-\omega_j)f'(\phi_{i\tau_n}^{k+1})
+\lambda_{i\tau_n}^{k+1}f'(\phi_{i\tau_n}^{k+1})\Big]v_i(x)\,\ud{x}=0,\non\\
&\qquad\forall\,0\leq
k\leq N_n-1,\;\forall\,v_i\in H^1(\T),\; i = 1, 2,
\end{align}
where $\lambda_{i\tau_n}^{k+1}$ is given by \eqref{lagrange-multiplier}.
Further, it holds
\begin{align}\label{H2-bound-discrete}
&\Big\|\frac12\Delta\phi_{1\tau_n}^{k+1}+\frac12\Delta\phi_{2\tau_n}^{k+1}
+\frac12\Delta\phi_{i\tau_n}^{k+1}- \dfrac{1}{2}W'(\phi_{i\tau_n}^{k+1}) + \dfrac{1}{2}W'(1-\phi_{1\tau_n}^{k+1}-\phi_{2\tau_n}^{k+1})\non\\
&\qquad-\sum_{j=1}^2\gamma_{ij}(-\Delta)^{-1}\big(f(\phi_{j\tau_n}^{k+1})-\omega_j)f'(\phi_{i\tau_n}^{k+1})
-\lambda_{i\tau_n}^{k+1}f'(\phi_{i\tau_n}^{k+1})\Big\|_{L^2(\T)}\non\\
\leq&\dfrac{\|\phi_{i\tau_n}^{k+1}-\phi_{i\tau_n}^{k}\|_{L^2(\T)}}{\tau_n},
\qquad\qquad\forall\,0\leq k\leq N_n-1,\; i = 1, 2.
\end{align}

\end{lemma}

\begin{proof}

First of all, by \eqref{Lp-strong-convergence} and \eqref{positivity-1}, there exists
$N=N(\beta)\in\mathbb{N}$, such that $\forall n > N$ and $\forall t\in[0,T]$,
\begin{equation}\label{nonzero-condition}
\int_{\T}\big|f'(\phi_{i\tau_n}(t,x))\big|^2\,\ud{x}\geq\frac{7\beta}{8},\;\int_{\T}\big|f'(\phi_{i\tau_n}(t,
x))-f'(\phi_i(t, x))\big|^2\,\ud{x}\leq\frac{\beta}{16},
\end{equation}
which together with Proposition \ref{proposition-positivity-2} gives
\begin{align*}
&\int_{\T}\big|f'(\phi_{i\tau_n}(t, x))-f'(\xi_i(t,x))\big|^2\,\ud{x} \\
\leq & 2\int_{\T}\big|f'(\phi_{i\tau_n}(t,x))-f'(\phi_i(t,x))\big|^2\,\ud{x}+2\int_{\T}\big|f'(\phi_i(t,x))-f'(\xi_i(t,x))\big|^2\,\ud{x} \leq\frac{\beta}{4},
\end{align*}
As a consequence, by \eqref{def-discrete-sequence}, $\forall\,n\geq N$, $\forall\,0\leq k\leq N_n$,  it turns out that
\begin{align*}
&\int_{\T}f'(\phi_{i\tau_n}^k(x))f'(\xi_i(\tau_nk,x)\,\ud{x}\non\\
=&\int_{\T}\big|f'(\phi_{i\tau_n}^k(x)\big|^2\,\ud{x}+\int_{\T}f'(\phi_{i\tau_n}^k(x)\big[f'(\xi_i(\tau_nk,x)-f'(\phi_{i\tau_n}^k(x))\big]\,\ud{x}\non\\
=&\int_{\T}\big|f'(\phi_{i\tau_n}(\tau_nk,x))\big|^2\,\ud{x}
+\int_{\T}f'(\phi_{i\tau_n}(\tau_nk,x)\big[f'(\xi_i(\tau_nk,x)-f'(\phi_{i\tau_n}(\tau_nk,x))\big]\,\ud{x}\non\\
\geq&\int_{\T}\big|f'(\phi_{i\tau_n}(\tau_nk,x))\big|^2\,\ud{x}
-\Big\|f'(\phi_{i\tau_n}(\tau_nk,x))\Big\|_{L^2(\T)}
\Big\|f'(\phi_{i\tau_n}(\tau_nk,x))-f'(\xi_i(\tau_nk,x))\Big\|_{L^2(\T)}\non\\
\geq&\sqrt{\frac{7\beta}{8}}\big(\sqrt{\frac{7\beta}{8}}-\frac{\sqrt{\beta}}{2}\big) >\frac{\beta}{8}.
\end{align*}
Therefore, given each $\phi_{i\tau_n}^k\in H_{\omega_i}^1$, let us choose
\begin{equation}\label{test-function}
w_{i\tau_n}^k(x)=f'(\xi_i(\tau_nk, x)),
\end{equation}
which yields
\begin{equation}\label{positivity-main}
\int_{\T}f'(\phi_{i\tau_n}^{k+1}(x))w_{i\tau_n}^{k+1}(x)\,\ud{x}>\frac{\beta}{8},\qquad \forall\,n\geq N, \quad \forall\,0\leq
k\leq N_n-1.
\end{equation}

Consequently, we derive that $\phi_{i\tau_n}^{k+1}$ satisfies
the Euler-Lagrange equation (see the appendix for details)
\begin{align}\label{eqn:EulerLagrange}
0&=\int_{\T}\Big[\dfrac{\phi_{i\tau_n}^{k+1}-\phi_{i\tau_n}^k}{\tau_n}+ \dfrac{1}{2}W'(\phi_{i\tau_n}^{k+1}) - \dfrac{1}{2}W'(1-\phi_{1\tau_n}^{k+1}-\phi_{2\tau_n}^{k+1}) \Big]v_i(x)\,\ud{x}\non\\
&\qquad+\int_{\T}\Big[\sum_{j=1}^2\gamma_{ij}(-\Delta)^{-1}\big(f(\phi_{j\tau_n}^{k+1})-\omega_j)f'(\phi_{i\tau_n}^{k+1})
+\lambda_{i\tau_n}^{k+1}f'(\phi_{i\tau_n}^{k+1})\Big]v_i(x)\,\ud{x}\non\\
&\qquad+\int_{\T}\Big[\frac12\nabla\phi_{1\tau_n}^{k+1}+\frac12\nabla\phi_{2\tau_n}^{k+1}+\frac12\nabla\phi_{i\tau_n}^{k+1}\Big]\cdot \nabla{v}_i\,\ud{x}
\end{align}
where the corresponding Lagrange multiplier is given by
\begin{align}\label{lagrange-multiplier}
\lambda_{i\tau_n}^{k+1}=&-\dfrac{1}{{\int_{\T}f'(\phi_{i\tau_n}^{k+1})w_{i\tau_n}^{k+1}\,\ud{x}}}
\left[\frac{1}{\tau_n}\int_{\T}(\phi_{i\tau_n}^{k+1}-\phi_{i\tau_n}^{k})w_{i\tau_n}^{k+1}\,\ud{x}+
\frac12\int_{\T}\big(\nabla\phi_{1\tau_n}^{k+1}+\nabla\phi_{2\tau_n}^{k+1}+\nabla\phi_{i\tau_n}^{k+1}\big)\nabla w_{i\tau_n}^{k+1}\,\ud{x}\right]\non\\
&-\dfrac{1}{{\int_{\T}f'(\phi_{i\tau_n}^{k+1}) w_{i\tau_n}^{k+1}\,\ud{x}}}\int_{\T}
\big[W'(\phi_{i\tau_n}^{k+1})-W'(1-\phi_{1\tau_n}^{k+1}-\phi_{2\tau_n}^{k+1})\big] w_{i\tau_n}^{k+1}\,\ud{x}\non\\
&-\dfrac{1}{{\int_{\T}f'(\phi_{i\tau_n}^{k+1}) w_{i\tau_n}^{k+1}\,\ud{x}}}\int_{\T}\Big[\sum_{j=1}^2\gamma_{ij}(-\Delta)^{-1}\big(f(\phi_{j\tau_n}^{k+1})-\omega_j)f'(\phi_{i\tau_n}^{k+1})\Big] w_{i\tau_n}^{k+1}\,\ud{x}
\end{align}

Meanwhile, note that
$(\phi_{i\tau_n}^{k+1}-\phi_{i\tau_n}^{k})/{\tau_n}$,
$W'(\phi_{i\tau_n}^{k+1})$, $W'(1-\phi_{1\tau_n}^{k+1}-\phi_{2\tau_n}^{k+1})$,
$(-\Delta)^{-1}\big(f(\phi_{j\tau_n}^{k+1})-\omega_j\big)f'(\phi_{i\tau_n}^{k+1})$,
$\lambda_{i\tau_n}^{k+1}f'(\phi_{i\tau_n}^{k+1})$ are all in $L^2(\T)$, hence
$\Delta\phi_{1\tau_n}^{k+1}+\Delta\phi_{2\tau_n}^{k+1}+\Delta\phi_{i\tau_n}^{k+1}\in L^2(\T)$ and
\eqref{equation-discrete-EL} is derived.

\smallskip

To proceed further, for sufficiently small $\eps>0$ we denote $\eta_{i\eps}\in H^2(\T)$ the unique solution (see for instance, \cite[Proposition 7.1]{Brezis_Book2011}) to the elliptic problem
\begin{align*}
\eta_{i\eps}-\eps\Delta\eta_{i\eps}&=\frac12\Delta\phi_{1\tau_n}^{k+1}+\frac12\Delta\phi_{2\tau_n}^{k+1}
+\frac12\Delta\phi_{i\tau_n}^{k+1}-W'(\phi_{i\tau_n}^{k+1})+W'(1-\phi_{1\tau_n}^{k+1}-\phi_{2\tau_n}^{k+1})\\
&\qquad-\sum_{j=1}^2\gamma_{ij}(-\Delta)^{-1}\big(f(\phi_{j\tau_n}^{k+1})-\omega_j)f'(\phi_{i\tau_n}^{k+1})
-\lambda_{i\tau_n}^{k+1}f'(\phi_{i\tau_n}^{k+1}).
\end{align*}
It follows from classical result (see for instance, \cite[Proposition 7.2]{Brezis_Book2011}) that as $\eps\rightarrow 0^+$
\begin{align}\label{Hille-Yosida}
\eta_{i\eps} &\longrightarrow \frac12\Delta\phi_{1\tau_n}^{k+1}+\frac12\Delta\phi_{2\tau_n}^{k+1}
+\frac12\Delta\phi_{i\tau_n}^{k+1}-W'(\phi_{i\tau_n}^{k+1})+W'(1-\phi_{1\tau_n}^{k+1}-\phi_{2\tau_n}^{k+1})\non\\
&\qquad\quad-\sum_{j=1}^2\gamma_{ij}(-\Delta)^{-1}\big(f(\phi_{j\tau_n}^{k+1})-\omega_j)f'(\phi_{i\tau_n}^{k+1})
-\lambda_{i\tau_n}^{k+1}f'(\phi_{i\tau_n}^{k+1})
\qquad\mbox{in } L^2(\T) .
\end{align}
Moreover, by choosing $v_i=\eta_{i\eps}$ in \eqref{equation-discrete-EL}
we get after integration by parts that
\begin{equation}\label{Schwarz-inequality}
\|\eta_{i\eps}\|_{L^2(\T)}^2+\eps\|\nabla\eta_{i\eps}\|_{L^2(\T)}^2=-\int_{\T}\dfrac{\phi_{i\tau_n}^{k+1}-\phi_{i\tau_n}^{k}}{\tau_n}\eta_{i\eps}(x)\,\ud{x}
\leq\dfrac{\|\phi_{i\tau_n}^{k+1}-\phi_{i\tau_n}^{k}\|_{L^2(\T)}}{\tau_n}\|\eta_{i\eps}\|_{L^2(\T)}.
\end{equation}
Hence \eqref{H2-bound-discrete} is proved by combining
\eqref{Hille-Yosida} and \eqref{Schwarz-inequality}.
\end{proof}

\begin{remark}\label{remark-approximation}
The main motivation to use the approximate $\xi_i$ given in Proposition \ref{proposition-positivity-2} is due to the term
\[
\int_{\T}\nabla\phi_{j\tau_n}^{k+1}\nabla w_{i\tau_n}^{k+1}=
6\int_{\T}\nabla\phi_{j\tau_n}^{k+1}\nabla\big[\xi_i(\tau_nk+\tau_n)-\xi_i^2(\tau_nk+\tau_n)\big]
\]
in the Lagrange multiplier \eqref{lagrange-multiplier}. If we simply use $\phi_i$ instead of $\xi_i$, the RHS above might not be integrable.
\end{remark}

From now on in this Section and the Appendix, when we say `` for $n>N$ ", it is always the $N = N(\beta)$ given in Lemma \ref{lemma-discrete-EL}.

For each fixed $n\geq N$, based on the Lagrange multipliers $\lambda_{i\tau_n}^k$, $i=1,2$, we introduce a piecewise constant interpolation $t\in [0, T]\mapsto\big(\lambda_{1\tau_n}(t), \lambda_{2\tau_n}(t)\big)$ by
\begin{equation}\label{def-piecewise-interpolation-2}
\lambda_{i\tau_n}(t)=\lambda_{i\tau_n}^k,
\quad\mbox{for }  \tau_n k\leq
t< \tau_n(k+1),
\end{equation}
where $\lambda_{i\tau_n}^k$ are given in \eqref{lagrange-multiplier}. Then for the piecewise-constant interpolation functional sequence $\{(\phi_{1\tau_n}, \phi_{2\tau_n})\}$ ($n\geq N$) defined in (\ref{def-piecewise-interpolation}), one may further retrieve the following {\it{a priori}} estimates.

\begin{lemma}\label{lemma-H2-integrability}
There exists a constant $C>0$ that may only depend on $T$, $\phi_{10}$, $\phi_{20}$, $\omega_1$, $\omega_2$, $\beta$, and $\gamma_{ij}$ ($i,j=1,2$), such that along the sequence $\{\tau_n\}$, it holds  that for $\forall\, n\geq N$
\begin{equation}
\int_{0}^{T}\|\phi_{i\tau_n}(t)\|_{H^2(\T)}^2\,\ud{t}\leq C, \qquad i=1,2.\label{H2-integrability}
\end{equation}
\end{lemma}
\begin{proof}
Consider \eqref{H2-bound-discrete} for any fixed $n\geq N$. Summing over $k$ from $0$ to $N_n-2$, we get from \eqref{discrete-inequality} that
\begin{align}\label{L2-integrability}
&\int_{\tau_n}^{N_n\tau_n}
\Big\|\frac12\Delta\phi_{1\tau_n}(t)+\frac12\Delta\phi_{2\tau_n}(t)
+\frac12\Delta\phi_{i\tau_n}(t) - W'(\phi_{i\tau_n}(t))+W'(1-\phi_{1\tau_n}(t)-\phi_{2\tau_n}(t))\non\\
&\qquad\quad-\sum_{j=1}^2\gamma_{ij}(-\Delta)^{-1}\Big(f(\phi_{j\tau_n}(t))-\omega_j\Big)f'(\phi_{i\tau_n}(t))
-\lambda_{i\tau_n}(t) f'(\phi_{i\tau_n}(t))\Big\|_{L^2(\T)}^2\,\ud{t}\non\\
\leq&\displaystyle\sum_{k=0}^{N_n-2}\int_{(k+1)\tau_n}^{(k+2)\tau_n}
\Big\|\frac12\Delta\phi_{1\tau_n}^{k+1}+\frac12\Delta\phi_{2\tau_n}^{k+1}
+\frac12\Delta\phi_{i\tau_n}^{k+1}-W'(\phi_{i\tau_n}^{k+1})+W'(1-\phi_{1\tau_n}^{k+1}-\phi_{2\tau_n}^{k+1})\non\\
&\qquad\qquad\quad-\sum_{j=1}^2\gamma_{ij}(-\Delta)^{-1}\big(f(\phi_{j\tau_n}^{k+1})-\omega_j)f'(\phi_{i\tau_n}^{k+1})
-\lambda_{i\tau_n}^{k+1}f'(\phi_{i\tau_n}^{k+1})\Big\|_{L^2(\T)}^2\,\ud{t}\non\\
\leq&\displaystyle\sum_{k=0}^{N_n-2}\dfrac{\big\|\phi_{i\tau_n}^{k+1}-\phi_{i\tau_n}^{k}\big\|_{L^2(\T)}^2}{\tau_n} \leq 2E(\phi_{1\tau_n}^0, \phi_{2\tau_n}^0)= 2E(\phi_{10}, \phi_{20}).
\end{align}
From \eqref{HLS-1}, H\"{o}lder's inequality and
\eqref{uniform-H1-bound}, we obtain that for $\forall\,n\geq N$, $\forall\, t\in[0,T]$ and $j=1,2$
\begin{align}\label{estimate-1}
& \big \|(-\Delta)^{-1}\big(f(\phi_{j\tau_n}(t))-\omega_j\big)f'(\phi_{i\tau_n}(t))\big\|_{L^2(\T)} \non\\
\leq &\big\|G\ast(f(\phi_{j\tau_n}(t))-\omega_j)\big\|_{L^6(\T)}\|f'(\phi_{i\tau_n}(t))\|_{L^3(\T)} \non\\
\leq & \  C\|f(\phi_{j\tau_n}(t))-\omega_j\|_{L^\frac{6}{5}(\T)}\big(\|\phi_{i\tau_n}(t)\|_{H^1(\T)}^3+1\big) \non\\
\leq & C(\phi_{10}, \phi_{20}, \omega_1, \omega_2),
\end{align}
Furthermore, in \eqref{lagrange-multiplier} let us denote
\begin{align*}\label{lagrange-multiplier-part-2}
\tilde{\lambda}_{i\tau_n}^{k+1}=-\frac{\frac{1}{\tau_n}\int_{\T}(\phi_{i\tau_n}^{k+1}-\phi_{i\tau_n}^{k})w_{i\tau_n}^{k+1}\,\ud{x}}
{\int_{\T}f'(\phi_{i\tau_n}^{k+1})w_{i\tau_n}^{k+1}\,\ud{x}}
\end{align*}
It is easy to infer from \eqref{HLS-1}, \eqref{estimates-discrete},\eqref{positivity-main} that for $\forall\,n\geq N, 0\leq k\leq N_n-1$, it holds
\begin{align*}
\big|\lambda_{i\tau_n}^{k+1}-\tilde{\lambda}_{i\tau_n}^{k+1}\big|
\leq&\dfrac{C}{\beta}\Big[\big(\|\nabla\phi_{1\tau_n}^{k+1}\|_{L^2}+\|\nabla\phi_{2\tau_n}^{k+1}\|_{L^2}\big)\|\nabla w_{i\tau_n}^{k+1}\|_{L^2}\\
&\quad+\big(\|W'(\phi_{i\tau_n}^{k+1})\|_{L^2}+\|W'(1-\phi_{1\tau_n}^{k+1}-\phi_{2\tau_n}^{k+1})\|_{L^2}\big) \| w_{i\tau_n}^{k+1}\|_{L^2}\\
&\qquad+\sum_{j=1}^2|\gamma_{ij}|\cdot \big\|G\ast(f(\phi_{j\tau_n}^{k+1})-\omega_j)\big\|_{L^6}\|f'(\phi_{i\tau_n}^{k+1})\|_{L^3}\| w_{i\tau_n}^{k+1}\|_{L^2}\Big]\\
\leq & C(\phi_{10}, \phi_{20}, \omega_1, \omega_2, \beta, \gamma_{i1}, \gamma_{i2}).
\end{align*}
Using \eqref{positivity-main} we have
\begin{equation*}
|\tilde{\lambda}_{i\tau_n}^{k+1}|\leq\frac{8}{\beta}\frac{\|\phi_{\tau_n}^{k+1}-\phi_{\tau_n}^{k}\|_{L^2}\| w_{i\tau_n}^{k+1}\|_{L^2}}{\tau_n}
\leq
C(\beta)\frac{\|\phi_{\tau_n}^{k+1}-\phi_{\tau_n}^{k}\|_{L^2}}{\tau_n}.
\end{equation*}
Henceforth we obtain
\begin{align}\label{estimate-2}
\int_{\tau_n}^{N_n\tau_n}\big\|\lambda_{i\tau_n}f'(\phi_{i\tau_n}(t))\big\|_{L^2(\T)}^2\,\ud{t}
&\leq\displaystyle\sum_{k=0}^{N_n-2}\int_{(k+1)\tau_n}^{(k+2)\tau_n}2\big(|\lambda_{i\tau_n}^k-\tilde{\lambda}_{i\tau_n}^k|^2+|\tilde{\lambda}_{i\tau_n}^k|^2\big)
\big\|f'(\phi_{i\tau_n}^{k+1})\big\|_{L^2(\T)}^2\,\ud{t}\non\\
&\leq
C(\phi_{10}, \phi_{20}, \omega_1, \omega_2, \beta, \gamma_{i1},\gamma_{i2})+C(\beta)\displaystyle\sum_{k=0}^{N_n-2}\dfrac{\big\|\phi_{i\tau_n}^{k+1}-\phi_{i\tau_n}^{k}\big\|_{L^2(\T)}^2}{\tau_n}\non\\
&\leq C(T, \phi_{10}, \phi_{20}, \omega_1, \omega_2, \beta, \gamma_{i1}, \gamma_{i2}).
\end{align}
In all, summing up Young's inequality, \eqref{uniform-H1-bound},
\eqref{L2-integrability}, \eqref{estimate-1}, \eqref{estimate-2}, we conclude that
\begin{align}\label{Laplace-bound-1}
&\int_{\tau_n}^{N_n\tau_n}\big\|\Delta\phi_{1\tau_n}(t)+\Delta\phi_{2\tau_n}(t)+\Delta\phi_{i\tau_n}(t)\big\|_{L^2(\T)}^2\,\ud{t}\non\\
\leq & \ 16E(\phi_{10},\phi_{20})+16\int_{\tau_n}^{N_n\tau_n}\big\|W'(\phi_{i\tau_n}(t))+W'(1-\phi_{1\tau_n}(t)-\phi_{2\tau_n}(t))\big\|_{L^2(\T)}^2\,\ud{t} \non\\
&\qquad+16\sum_{j=1}^2|\gamma_{ij}|\int_{\tau_n}^{N_n\tau_n}\big\|(-\Delta)^{-1}\big(f(\phi_{j\tau_n}(t))-\omega_j\big)f'(\phi_{i\tau_n}(t)) \big\|_{L^2(\T)}^2\,\ud{t}\non\\
&\qquad+16\int_{\tau_n}^{N_n\tau_n}\big\|\lambda_{i\tau_n}f'(\phi_{i\tau_n}(t))\big\|_{L^2(\T)}^2\,\ud{t} \non\\
\leq & \  C(T, \phi_{10}, \phi_{20}, \omega_1, \omega_2, \beta, \gamma_{i1}, \gamma_{i2}).
\end{align}
Therefore, using Young's inequality and \eqref{Laplace-bound-1}, we have
\begin{align*}
&\int_{\tau_n}^{N_n\tau_n}\big\|\Delta\phi_{1\tau_n}(t)\big\|_{L^2(\T)}^2\,\ud{t} \\
= &\ \dfrac{1}{9}\int_{\tau_n}^{N_n\tau_n}\big\| 2 \big[ 2\Delta\phi_{1\tau_n}(t) + \Delta\phi_{2\tau_n}(t) \big] - \big[ 2\Delta\phi_{2\tau_n}(t) + \Delta\phi_{1\tau_n}(t) \big] \big\|_{L^2(\T)}^2\,\ud{t} \\
\leq &\ \frac{8}{9}\int_{\tau_n}^{N_n\tau_n}\big\| 2\Delta\phi_{1\tau_n}(t)+\Delta\phi_{2\tau_n}(t) \big\|_{L^2(\T)}^2\,\ud{t}
+\frac{2}{9}\int_{\tau_n}^{N_n\tau_n}\big\| 2\Delta\phi_{2\tau_n}(t)+\Delta\phi_{1\tau_n}(t)  \big\|_{L^2(\T)}^2\,\ud{t}\\
\leq &\ C(T, \phi_{10}, \phi_{20}, \omega_1, \omega_2, \beta, \gamma_{11}, \gamma_{12}, \gamma_{22}).
\end{align*}
and the estimate for $\int_{\tau_n}^{N_n\tau_n}\big\|\Delta\phi_{2\tau_n}(t)\big\|_{L^2(\T)}^2\,\ud{t}$ can be established in a similar manner, which together with \eqref{uniform-H1-bound}, and monotone convergence theorem leads to \eqref{H2-integrability}.
\end{proof}


\subsection{Convergence to the limit curve}\label{sec:ConvToLimitCurve}

After collecting all the {\it{a priori}} estimates and the Euler-Lagrange equation \eqref{equation-discrete-EL} established in the previous subsections, in this subsection we shall show that the limit curve $(\phi_1, \phi_2)$ retrieved in Proposition \ref{proposition-convergence} indeed solves the equation \eqref{GF}.

To begin with, for $\forall\,0<t< T$, $\forall\, n\geq N$, denote $\tilde{N}_n = \lfloor t/\tau_n \rfloor$. Summing $k$ from $0$ to $\tilde{N}_n - 1$ in \eqref{equation-discrete-EL}, using \eqref{def-discrete-sequence}, \eqref{def-piecewise-interpolation} and \eqref{def-piecewise-interpolation-2}, it is easy to check that
\begin{align}\label{equation-discrete-level}
&\int_{\T}\big[\phi_{i\tau_n}(t,x)-\phi_{i0}(x)\big]v_i(x)\,\ud{x}\non\\
=&\int_{\tau_n}^{\tilde{N}_n\tau_n+\tau_n}\int_{\T}\frac12\Big[\Delta\phi_{1\tau_n}(s,x)+\Delta\phi_{2\tau_n}(s,x)
+\Delta\phi_{i\tau_n}(s,x) - W'(\phi_{i\tau_n}(s,x))\Big]v_i(x)\,\ud{x}\ud{s}\non\\
&\qquad+\int_{\tau_n}^{\tilde{N}_n\tau_n+\tau_n}\int_{\T}\Big[\frac12 W'(1-\phi_{1\tau_n}-\phi_{2\tau_n})-\sum_{j=1}^2\gamma_{ij}(-\Delta)^{-1}(f(\phi_{j\tau_n}(s,x))-\omega_j)f'(\phi_{i\tau_n}(s,x))\Big]v_i(x)\,\ud{x}\ud{s}\non\\
&\qquad-\int_{\tau_n}^{\tilde{N}_n\tau_n+\tau_n}\int_{\T}\Big[\lambda_{i\tau_n}(s)f'(\phi_{i\tau_n}(s,x))\Big]v_i(x)\,\ud{x}\ud{s},
\qquad\qquad i=1,2.
\end{align}
By equation \eqref{equation-discrete-level}, Proposition \ref{proposition-convergence}, Lemmas \ref{estimates-discrete}-\ref{lemma-H2-integrability}, we are ready to prove the main theorem regarding to the existence result.

\begin{definition}
We call $(\phi_1(t, x), \phi_2(t, x))$ a global weak solution to problem \eqref{GF}-\eqref{eqn:Volume}, if for $\forall T>0$, $(\phi_1, \phi_2)$ satisfies
\begin{align*}
\phi_i&\in C([0, T]; L^p(\T))\cap L^\infty(0, T; H^1_{\omega_i})\cap
L^2(0, T;H^2(\T)),\quad i=1,2,
\end{align*}
the initial condition \eqref{IC}, and the volume constraint \eqref{eqn:Volume}, for $\forall\,t\in [0, T]$. Further, for $\forall t\in (0, T)$, any test functions $w_1, w_2\in L^2(\T)$, it holds
\begin{align}\label{def:weaksolution}
\dfrac{\ud}{\ud{t}}\int_{\T}\phi_i(t,x)w_i(x)\,\ud{x}\ = &\int_{\T}\bigg[\frac12\Delta\phi_1+\frac12\Delta\phi_2+\frac12\Delta\phi_i- \frac12 W'(\phi_i)+ \frac12 W'(1-\phi_1-\phi_2)\non\\
&\qquad\quad-\sum_{j=1}^2\gamma_{ij}(-\Delta)^{-1}\big(f(\phi_j)-\omega_j\big)f'(\phi_i)-\lambda_i{f}'(\phi_i)\bigg]w_i(x)\,\ud{x}
\end{align}
in the distributional sense in $(0, T)$ for $i=1,2$.
\end{definition}

\begin{theorem}\label{theorem-main}
For any $\omega_1, \omega_2\in\RR$ that satisfy \eqref{tech-assumption}, $(\phi_{10}, \phi_{20})\in H^1_{\omega_1}\times H^2_{\omega_2}$, there exists a unique global weak solution $(\phi_1, \phi_2)$ to problem \eqref{GF}-\eqref{IC} with volume constraint \eqref{eqn:Volume}. Further, the free energy $E(\phi_1(t), \phi_2(t))$ is decreasing as time evolves.
\end{theorem}

\begin{proof}

\noindent\textbf{Existence:} To begin with, using \eqref{uniform-H1-bound}, \eqref{H2-integrability}, we can further get up to a subsequence (for simplicity we shall not distinguish the sequence $\{\tau_n\}$ and its subsequence now and later) that
\begin{align}
&\phi_{i\tau_n}(t) \overset{\ast}{\rightharpoonup} \phi_i(t), \quad\mbox{weak} \ast \mbox{ in
} L^\infty(0, T; H^1(\T)), \non\\
&\phi_{i\tau_n}(t)\rightharpoonup\phi_i(t), \quad\mbox{weakly} \mbox{ in
} L^2(0, T; H^2(\T)),
\label{weak-star-2}
\end{align}

From now on in the proof, we will take the test functions $v_i(x) \in H^1(\T)$ with better regularity than $w_i(x) \in L^2(\T)$ in (\ref{def:weaksolution}), then by a classical density argument, it is easy to check that \eqref{GF-integral-form} below is valid for any test function $w_i\in L^2(\T)$.

As a consequence, passing $n\rightarrow+\infty$ in \eqref{equation-discrete-level} we get by \eqref{L2-strong-convergence} that
\begin{equation}\label{convergence-temporal}
\int_{\T}\big[\phi_{i\tau_n}(t,x)-\phi_{i0}(x)\big]v_i(x)\,\ud{x}\rightarrow\int_{\T}\big[\phi_i(t,x)-\phi_{i0}(x)\big]v_i(x)\,\ud{x}.
\end{equation}
Meanwhile, using \eqref{uniform-H1-bound},
\eqref{L2-strong-convergence} and dominated convergence theorem we get
\begin{align}
\int_{\tau_n}^{\tilde{N}_n\tau_n+\tau_n}\int_{\T}W'(\phi_{i\tau_n}(s,x))v_i(x)\,\ud{x}\ud{s}
&\rightarrow\int_0^t\int_{\T}W'(\phi_i(s,x))v_i(x)\,\ud{x}\ud{s}, \label{convergence-space-2}\\
\int_{\tau_n}^{\tilde{N}_n\tau_n+\tau_n}\int_{\T}W'(1-\phi_{1\tau_n}-\phi_{2\tau_n})v_i(x)\,\ud{x}\ud{s}
&\rightarrow\int_0^t\int_{\T}W'(1-\phi_1-\phi_2)v_i(x)\,\ud{x}\ud{s} \label{convergence-space-2-extra}
\end{align}
It is worth pointing out that in \eqref{convergence-space-2} and \eqref{convergence-space-2-extra} the
bordering time integral can be ignored because the above bordering
time integrand is bounded in $L^\infty(0, T; L^1(\T))$.

\smallskip

To proceed, note that \eqref{H2-integrability} and \eqref{weak-star-2} together imply that
\begin{equation}\label{laplace-term-integrable}
\Delta\phi\in L^2(0, T; L^2(\T)),
\end{equation}
and henceforth we know from \eqref{weak-star-2} and
\eqref{laplace-term-integrable} that
\begin{align}\label{convergence-space-3}
\int_{\tau_n}^{\tilde{N}_n\tau_n+\tau_n}\int_{\T}\Delta\phi_{i\tau_n}(s,x)v_i(x)\,\ud{x}\ud{s}
\rightarrow\int_0^t\int_{\T}\Delta\phi_i(s,x)v_i(x)\,\ud{x}\ud{s}.
\end{align}
It is worth mentioning that bordering time integrals can be
neglected for the same reason as in \eqref{convergence-space-2}.

\smallskip

Next, using dominated convergence theorem, we derive from
\eqref{HLS-1}, \eqref{uniform-H1-bound} and
\eqref{Lp-strong-convergence} that
\begin{align*}
&\bigg|\int_0^t\int_{\T}(-\Delta)^{-1}(f(\phi_{j\tau_n})-\omega_j)f'(\phi_{i\tau_n})v_i(x) \ \ud{x}\ud{s} - \int_0^t\int_{\T}(-\Delta)^{-1}(f(\phi_j)-\omega_j)f'(\phi_i)v_i(x)  \ \ud{x}\ud{s} \bigg|\\
\leq & \int_0^t\int_{\T}\Big|(-\Delta)^{-1}(f(\phi_{j\tau_n})-f(\phi_j))f'(\phi_{i\tau_n})v_i(x)\Big|  + \Big|(-\Delta)^{-1}(f(\phi_j)-\omega_j)\big[f'(\phi_{i\tau_n})-f'(\phi_i)\big]v_i(x)\Big|\,\ud{x}\ud{s}\\
\leq & \int_0^t\left\|G\ast\big(f(\phi_{j\tau_n}(s))-f(\phi_j(s))\big)\right\|_{L^6(\T)}\big\|f'(\phi_{i\tau_n}(s))\big\|_{L^3(\T)}\|v_i\|_{L^2(\T)}\,\ud{s}\\
&\qquad+\int_0^t\left\|G\ast\big(f(\phi_j(s))-\omega_j\big)\right\|_{L^6(\T)}\big\|f'(\phi_{i\tau_n}(s))-f'(\phi_i(s))\big\|_{L^2(\T)}\|v_i\|_{L^3(\T)}\,\ud{s}\\
\leq & C\int_0^t\big\|f(\phi_{j\tau_n}(s))-f(\phi_j(s))\big\|_{L^{\frac65}(\T)}\,\ud{s} + C\int_0^t\big\|f(\phi_j(s))-\omega_j\big\|_{L^{\frac65}(\T)}\big\|f'(\phi_{i\tau_n}(s))-f'(\phi_i(s))\big\|_{L^2(\T)}\,\ud{s}\\
\rightarrow & 0,
\end{align*}
for $j=1,2$. Note that the $H^1$ regularity of the test function $v_i(x)$ is to bound $\|v_i\|_{L^3(\T)}$ in the above estimate. Hence we get (after ignoring borderline time integrands) that
\begin{align}\label{convergence-space-4}
&\sum_{j=1}^2\gamma_{ij}\int_{\tau_n}^{\tilde{N}_n\tau_n+\tau_n}\int_{\T}(-\Delta)^{-1}(f(\phi_{j\tau_n}(s,x))-\omega_j)f'(\phi_{i\tau_n}(s,x))v_i(x)\,\ud{x}\ud{s}\non\\
\rightarrow & \sum_{j=1}^2\gamma_{ij}\int_0^t\int_{\T}(-\Delta)^{-1}(f(\phi_j(s,x))-\omega_j)f'(\phi_i(s,x))v_i(x)\,\ud{x}\ud{s}.
\end{align}
Finally, let us consider the convergence of the last term on the RHS of \eqref{equation-discrete-level}. By \eqref{uniform-H1-bound}, \eqref{Lp-strong-convergence}, and the dominated convergence theorem we get
\begin{equation}\label{strong-converence}
\int_{\T}f'(\phi_{i\tau_n}(s,
x))v_i(x)\,\ud{x}\rightarrow\int_{\T}f'(\phi_i(s, x))v_i(x)\,\ud{x}
\;\;\mbox{strongly in } L^2(0, T).
\end{equation}
Moreover, in a similar manner as in the proof of Lemma \ref{lemma-H2-integrability} we obtain
\begin{align}
\int_{\tau_n}^{T-\tau_n}\big|\lambda_{i\tau_n}(s)\big|^2\,\ud{t}
&\leq\displaystyle\sum_{k=0}^{N_n-2}\int_{(k+1)\tau_n}^{(k+2)\tau_n}2\big(|\lambda_{i\tau_n}^k-\tilde{\lambda}_{i\tau_n}^k|^2
+|\tilde{\lambda}_{i\tau_n}^k|^2\big)\,\ud{t}\non\\
&\leq C(T,\phi_{10},\phi_{20},\omega_1,\omega_2,\beta,\gamma_{11},\gamma_{12},\gamma_{22})
+C(\beta)\displaystyle\sum_{k=0}^{N_n-2}\dfrac{\big\|\phi_{i\tau_n}^{k+1}-\phi_{i\tau_n}^{k}\big\|_{L^2(\T)}^2}{\tau_n}\non\\
&\leq C(T,\phi_{10},\phi_{20},\omega_1,\omega_2,\beta,\gamma_{11},\gamma_{12},\gamma_{22}),
\end{align}
which together with monotone convergence theorem implies
\begin{equation}\label{weak-convergence}
\lambda_{i\tau_n}(s)\rightarrow\lambda_i(s)\;\;\mbox{weakly in }
L^2(0, T), \;\;\mbox{where }\, \lambda_i(s)\in L^2(0, T)
\end{equation}

In all, \eqref{strong-converence}-\eqref{weak-convergence} give
\begin{align}\label{convergence-space-5}
\int_{\tau_n}^{\tilde{N}_n\tau_n+\tau_n}\int_{\T}\Big[\lambda_{i\tau_n}(s)f'(\phi_{i\tau_n}(s,x))\Big]v_i(x)\,\ud{x}\ud{s}
\rightarrow\int_0^t\int_{\T}\big[\lambda_i(s)f'(\phi_i(s,x))\big]v_i(x)\,\ud{x}\ud{s}.
\end{align}

\smallskip

In conclusion, summing up the convergence results in \eqref{convergence-temporal}, \eqref{convergence-space-2}, \eqref{convergence-space-2-extra}, \eqref{convergence-space-3}, \eqref{convergence-space-4}, \eqref{convergence-space-5}, we manage to establish $\forall\,0<t< T$ the equation
\begin{align}\label{GF-integral-form}
\int_{\T}\big[\phi_i(t,x)-\phi_{i0}(x)\big]v_i(x)\,\ud{x} = & \int_0^t\int_{\T}\bigg[\frac12\Delta\phi_1+\frac12\Delta\phi_2+\frac12\Delta\phi_i- \frac12 W'(\phi_i)+\frac12 W'(1-\phi_1-\phi_2)\non\\
&-\sum_{j=1}^2\gamma_{ij}(-\Delta)^{-1}\big(f(\phi_j)-\omega_j\big)f'(\phi_i)-\lambda_i{f}'(\phi_i)\bigg]v_i(x)\,\ud{x}
\ud{s},
\end{align}
Hence we obtain a weak solution to the problem \eqref{GF}-\eqref{IC} in its integral form, which is equivalent to \eqref{GF} by \cite[Lemma 1.1, Chapter 3]{Temam_Book2001}. Further, to establish \eqref{lambda}, we multiply both sides of \eqref{GF} with $f'(\phi_i)$ and then integrate over $\T$. Note that \eqref{eqn:Volume} and \eqref{positivity-1} can be utilized.

\bigskip

\noindent\textbf{Uniqueness:} Suppose there are two global weak solutions, namely $(\phi_1,\phi_2)$, $(\phi_1^{\ast},\phi_2^{\ast})$ to problem \eqref{GF}-\eqref{IC}.
First of all, we know that
\begin{align*}
&\phi_i, \phi_i^\ast \in L^\infty(0, T; H^1(\T))\cap L^2(0, T; H^2(\T)), \quad 1\leq i\leq 2,\\
&\int_{\T}|f'(\phi_i)(t,x)|^2\,\ud{x}\geq\beta_i>0, \quad \int_{\T}|f'(\phi_i^\ast)(t,x)|^2\,\ud{x}=\beta_i^\ast>0, \quad\forall t\in [0, T].
\end{align*}

Let us define $\tilde{\phi_i}=\phi_i-\phi_i^\ast$, $i=1, 2$, then $(\tilde{\phi}_1,\tilde{\phi}_2)$ satisfies
\begin{align}\label{equation-difference}
\partial_t\tilde{\phi_i}&=\Delta\tilde{\phi}_i+\frac{\Delta\tilde{\phi}_j}{2}-\frac12\Big(\frac{\partial{W}_T}{\partial\phi_i}-\frac{\partial{W}_T}{\partial\phi_i^\ast}\Big)
-\sum_{k=1}^2\gamma_{ik}\big[(-\Delta)^{-1}(f(\phi_k)-\omega_k)f'(\phi_i)-(-\Delta)^{-1}(f(\phi_k^\ast)-\omega_k)f'(\phi_i^\ast)\big]\non\\
&\qquad-\lambda_i(t)f'(\phi_i)+\lambda_i^{\ast}f'(\phi_i^\ast), \qquad 1\leq i, j\leq 2,\; i\neq j.
\end{align}
subject to periodic boundary condition and the initial conditions
\begin{equation}\label{IC-difference-equ}
\tilde{\phi}_i(0, x)=0, \qquad i=1, 2.
\end{equation}
Multiplying equation \eqref{equation-difference} with $2\tilde{\phi}_i$, and summing over $i$ from $1$ to $2$, after integrating over $\T$, we get
\begin{align}
&\frac{\ud}{\ud{t}}\int_{\T}\big(|\tilde{\phi}_1(t,x)|^2+|\tilde{\phi}_2(t,x)|^2\big)\,\ud{x} \non\\
&=-2\int_{\T}\big(|\nabla\tilde{\phi}_1(t,x)|^2+|\nabla\tilde{\phi}_2(t,x)|^2\big)\,\ud{x}-2\int_{\T}\nabla\tilde{\phi}_1\cdot\nabla\tilde{\phi}_2(t,x)\,\ud{x}
-\sum_{i=1}^2\int_{\T}\Big[\frac{\partial{W}_T}{\partial\phi_i}-\frac{\partial{W_T}}{\partial\phi_i^\ast}\Big]\tilde{\phi}_i(t,x)\,\ud{x}\non\\
&\qquad-2\sum_{i,k=1}^2\int_{\T}\gamma_{ik}(-\Delta)^{-1}\big(f(\phi_k)-f(\phi_k^\ast)\big)f'(\phi_i)\tilde{\phi}_i(t,x)\,\ud{x}\non\\
&\qquad-2\sum_{i,k=1}^2\int_{\T}\gamma_{ik}(-\Delta)^{-1}\big(f(\phi_k^\ast)-\omega_k\big)\big[f'(\phi_i)-f'(\phi_i^\ast)\big]\tilde{\phi}_i(t,x)\,\ud{x}\non\\
&\qquad-\sum_{i=1}^2\lambda_i(t)\int_{\T}\big[f'(\phi_i)-f'(\phi_i^\ast)\big]\tilde{\phi}_i(t,x)\,\ud{x}
-\sum_{i=1}^2\big[\lambda_i(t)-\lambda_i^\ast(t)\big]\int_{\T}f'(\phi_i^\ast)\tilde{\phi}_i(t,x)\,\ud{x}  \non\\
&\defeq-2\big\|\nabla\tilde{\phi}_1(t,\cdot)\big\|_{L^2}^2-2\big\|\nabla\tilde{\phi}_2(t,\cdot)\big\|_{L^2}^2+I_1+\cdots+I_6.
\end{align}
We shall estimate $I_1, \cdots, I_6$ individually. First, it is easy to check
\begin{align}\label{I-1}
I_1\leq\big\|\nabla\tilde{\phi}_1(t,\cdot)\big\|_{L^2}^2+\big\|\nabla\tilde{\phi}_2(t,\cdot)\big\|_{L^2}^2.
\end{align}
Next, using mean value theorem, interpolation inequality and Young's inequality we know that
\begin{align}\label{I-2}
I_2&\leq C\big\|\phi_1^2+\phi_2^2+{\phi_1^{\ast}}^2+{\phi_2^{\ast}}^2+1\big\|_{L^2}\big(\|\tilde{\phi}_1\|_{L^4}^2+\|\tilde{\phi}_2\|_{L^4}^2\big)\non\\
&\leq C\big(\|\tilde{\phi}_1(t,\cdot)\|_{L^2}^2+\|\tilde{\phi}_2(t,\cdot)\|_{L^2}^2\big)
+\frac{1}{4}\big(\|\nabla\tilde{\phi}_1(t,\cdot)\|_{L^2}^2+\|\nabla\tilde{\phi}_2(t,\cdot)\|_{L^2}^2\big)
\end{align}
To proceed further, by mean value theorem and \eqref{HLS-1} we have
\begin{align}\label{I-3}
I_3&\leq 2\sum_{i,k=1}^2|\gamma_{ik}|\big\|(-\Delta)^{-1}(f(\phi_k)-f(\phi_k^\ast))\big\|_{L^6(\T)}\|f'(\phi_i)\|_{L^3(\T)}\|\tilde{\phi}_i\|_{L^2(\T)}\non\\
&\leq \sum_{i,k=1}^2C\big\|(f(\phi_k)-f(\phi_k^\ast))\big\|_{L^{\frac65}(\T)}\|\tilde{\phi}_i\|_{L^2(\T)}\non\\
&\leq \sum_{i,k=1}^2C\|f'(\eta_k)\|_{L^3}\|\tilde{\phi}_i\|_{L^2}^2 \qquad\mbox{where } \eta_k\in (\phi_k, \phi_k^\ast) \non\\
&\leq C\big(\|\tilde{\phi}_1(t,\cdot)\|_{L^2}^2+\|\tilde{\phi}_2(t,\cdot)\|_{L^2}^2\big).
\end{align}
At the same time, by mean value theorem, \eqref{HLS-1} and Young's inequality, one can show that
\begin{align}\label{I-4}
I_4&\leq 2\sum_{i,k=1}^2|\gamma_{ik}|\big\|(-\Delta)^{-1}(f(\phi_k^\ast)-\omega_k)\big\|_{L^6}\|f'(\phi_i)-f'(\phi_i^\ast)\|_{L^2}\|\tilde{\phi}_i\|_{L^3}\non\\
&\leq\sum_{i,k=1}^2C\|f(\phi_k^\ast)-\omega_k\big\|_{L^{\frac65}}\|f''(\eta_i)\|_{L^6}\|\tilde{\phi}_i\|_{L^3}\|\tilde{\phi}_i\|_{L^3}
 \qquad\mbox{where } \eta_i\in (\phi_i, \phi_i^\ast)\non\\
&\leq\sum_{i=1}^2C\Big(\|\tilde{\phi}_i\|_{L^2}^{\frac12}\|\nabla\tilde{\phi}_i\|_{L^2}^{\frac12}+\|\tilde{\phi}_i\|_{L^2}\Big)^2\non\\
&\leq C\big(\|\tilde{\phi}_1(t,\cdot)\|_{L^2}^2+\|\tilde{\phi}_2(t,\cdot)\|_{L^2}^2\big)
+\frac{1}{4}\big(\|\nabla\tilde{\phi}_1(t,\cdot)\|_{L^2}^2+\|\nabla\tilde{\phi}_2(t,\cdot)\|_{L^2}^2\big)
\end{align}
We proceed to estimate $I_5$ as follows
\begin{align}\label{I-5}
I_5&\leq\sum_{i=1}^2|\lambda_i(t)|\big\|f'(\phi_i)-f'(\phi_i^\ast)\big\|_{L^2}\|\tilde{\phi}_i\|_{L^2}\non\\
&\leq\sum_{i=1}^2|\lambda_i(t)|\big\|f''(\eta_i)\big\|_{L^6}\|\tilde{\phi}_i\|_{L^3}\|\tilde{\phi}_i\|_{L^2}\qquad\mbox{where } \eta_i\in (\phi_i, \phi_i^\ast)\non\\
&\leq\sum_{i=1}^2C|\lambda_i(t)|\|\tilde{\phi}_i\|_{L^2}\Big(\|\tilde{\phi}_i\|_{L^2}^{\frac12}\|\nabla\tilde{\phi}_i\|_{L^2}^{\frac12}+\|\tilde{\phi}_i\|_{L^2}\Big)\non\\
&\leq C(1+|\lambda_1(t)|^2+|\lambda_2(t)|^2)\left(\|\tilde{\phi}_1(t,\cdot)\|_{L^2}^2+\|\tilde{\phi}_2(t,\cdot)\|_{L^2}^2\right)\non\\
&\qquad+\frac{1}{4}\left(\|\nabla\tilde{\phi}_1(t,\cdot)\|_{L^2}^2+\|\nabla\tilde{\phi}_2(t,\cdot)\|_{L^2}^2\right)
\end{align}
Finally, to deal with $I_6$, first we estimate $|\lambda_i(t)-\lambda_i^\ast(t)|$ for $i=1,2$. By \eqref{lambda}, we see that
\begin{align*}
&\lambda_1(t)-\lambda_1^\ast(t)\\
=&\frac{-\int_{\T}|f'(\phi_1^\ast)|^2\ud{x}\int_{\T}|\nabla\phi_1|^2f''(\phi_1)\,\ud{x}
+\int_{\T}|f'(\phi_1)|^2\ud{x}\int_{\T}|\nabla\phi_1^\ast|^2f''(\phi_1^\ast)\,\ud{x}}{\int_{\T}|f'(\phi_1)|^2\ud{x}\int_{\T}|f'(\phi_1^\ast)|^2\ud{x}}\\
&\qquad-\frac{\int_{\T}|f'(\phi_1^\ast)|^2\ud{x}\int_{\T}(\nabla\phi_1\cdot\nabla\phi_2)f''(\phi_1)\,\ud{x}
-\int_{\T}|f'(\phi_1)|^2\ud{x}\int_{\T}(\nabla\phi_1^\ast\cdot\nabla\phi_2^\ast)f''(\phi_1^\ast)\,\ud{x}}{2\int_{\T}|f'(\phi_1)|^2\ud{x}\int_{\T}|f'(\phi_1^\ast)|^2\ud{x}}\\
&\qquad-\frac{\int_{\T}|f'(\phi_1^\ast)|^2\ud{x}\int_{\T}W'(\phi_1)f'(\phi_1)\,\ud{x}-\int_{\T}|f'(\phi_1)|^2\ud{x}
\int_{\T}W'(\phi_1^\ast)f'(\phi_1^\ast)\,\ud{x}}{2\int_{\T}|f'(\phi_1)|^2\ud{x}\int_{\T}|f'(\phi_1^\ast)|^2\ud{x}}\\
&\qquad+\frac{\int_{\T}|f'(\phi_1^\ast)|^2\ud{x}\int_{\T}W'(1-\phi_1-\phi_2)f'(\phi_1)\,\ud{x}-\int_{\T}|f'(\phi_1)|^2\ud{x}\int_{\T}W'(1-\phi_1^\ast-\phi_2^\ast)
f'(\phi_1^\ast)\,\ud{x}}{2\int_{\T}|f'(\phi_1)|^2\ud{x}\int_{\T}|f'(\phi_1^\ast)|^2\ud{x}}\\
&\qquad-\sum_{k=1}^2\frac{\int_{\T}|f'(\phi_1^\ast)|^2\ud{x}\int_{\T}\gamma_{1k}(-\Delta)^{-1}(f(\phi_k)-\omega_k)f'(\phi_1)
\,\ud{x}}{\int_{\T}|f'(\phi_1)|^2\ud{x}\int_{\T}|f'(\phi_1^\ast)|^2\ud{x}}\\
&\qquad+\sum_{k=1}^2\frac{\int_{\T}|f'(\phi_1)|^2\ud{x}\int_{\T}\gamma_{1k}(-\Delta)^{-1}(f(\phi_k^\ast)-\omega_k^\ast)f'(\phi_1^\ast)
\,\ud{x}}{\int_{\T}|f'(\phi_1)|^2\ud{x}\int_{\T}|f'(\phi_1^\ast)|^2\ud{x}}\\
\defeq & J_1+\cdots+J_6.
\end{align*}
Note that
\begin{align*}
J_1&=\frac{\int_{\T}(|f'(\phi_1)|^2-|f'(\phi_1^\ast)|^2)\,\ud{x}\int_{\T}|\nabla\phi_1|^2f''(\phi_1)\,\ud{x}}{\int_{\T}|f'(\phi_1)|^2\ud{x}\int_{\T}|f'(\phi_1^\ast)|^2\ud{x}}
\\
&\qquad+\frac{\int_{\T}|f'(\phi_1^\ast)|^2\,\ud{x}\int_{\T}\left( |\nabla\phi_1|^2f''(\phi_1)-|\nabla\phi_1^\ast|^2f''(\phi_1^\ast)\right)\,\ud{x}}{\int_{\T}|f'(\phi_1)|^2\ud{x}\int_{\T}|f'(\phi_1^\ast)|^2\ud{x}}\\
&\defeq J_{1a}+J_{1b},
\end{align*}
where using interpolation inequality we get
\begin{align*}
|J_{1a}|&\leq\frac{1}{\beta_1\beta_1^{\ast}}\int_{\T}\big|f'(\phi_1)-f'(\phi_1^\ast)\big|\big|f'(\phi_1)+f'(\phi_1^\ast)\big|\,\ud{x}\int_{\T}|f''(\phi_1)||\nabla\phi_1|^2\,\ud{x}\\
&\leq\frac{1}{\beta_1\beta_1^{\ast}}\int_{\T}|\tilde{\phi}_1||f''(\eta_1)|\big|f'(\phi_1)+f'(\phi_1^\ast)\big|\,\ud{x}\int_{\T}|f''(\phi_1)||\nabla\phi_1|^2\,\ud{x}
\qquad\mbox{where } \eta_1\in (\phi_1, \phi_1^\ast)\\
&\leq C\|\tilde{\phi}_1\|_{L^2(\T)}\|f''(\phi_1)\|_{L^6(\T)}\|\nabla\phi_1\|_{L^2(\T)}\|\nabla\phi_1\|_{L^3(T)}\\
&\leq C\|\tilde{\phi}_1\|_{L^2(\T)}\big(\|\nabla\phi_1\|_{L^2(\T)}^{\frac12}\|\Delta\phi_1\|_{L^2(\T)}^{\frac12}+\|\nabla\phi_1\|_{L^2(\T)}\big)\\
&\leq C\big(1+\|\Delta\phi_1\|_{L^2(\T)}^{\frac12}\big)\|\tilde{\phi}_1\|_{L^2(\T)},
\end{align*}
and
\begin{align*}
|J_{1b}|&\leq\frac{1}{\beta_1}\left|\int_{\T}\big(|\nabla\phi_1|^2-|\nabla\phi_1^\ast|^2\big)f''(\phi_1)\,\ud{x}
+\int_{\T}|\nabla\phi_1^\ast|^2\Big(f''(\phi_1)-f''(\phi_1^\ast)\Big)\,\ud{x}\right| \\
&\leq C\int_{\T}|\nabla\tilde{\phi}_1|\big(|\nabla\phi_1|+|\nabla\phi_1^\ast|\big)|f''(\phi_1)|\,\ud{x}+C\int_{\T}|\tilde{\phi}_1||\nabla\phi_1^\ast|^2\,\ud{x}\\
&\leq C\|f''(\phi_1)\|_{L^6}\big(\|\nabla\phi_1\|_{L^3}+\|\nabla\phi_1^\ast\|_{L^3}\big)\|\nabla\tilde{\phi}_1\|_{L^2}
+\|\tilde{\phi}_1\|_{L^2}\|\nabla\phi_1^\ast\|_{L^4}^2\\
&\leq C\big(1+\|\Delta\phi_1\|_{L^2}^{\frac12}+\|\Delta\phi_1^\ast\|_{L^2}^{\frac12}\big)\|\nabla\tilde{\phi}_1\|_{L^2}
+C(1+\|\Delta\phi_1^\ast\|_{L^2})\|\tilde{\phi}_1\|_{L^2}.
\end{align*}
In all, we have
\begin{align*}
|J_1|\leq C\Big(1+\|\Delta\phi_1\|_{L^2}^{\frac12}+\|\Delta\phi_1^\ast\|_{L^2}^{\frac12}\Big)\|\nabla\tilde{\phi}_1\|_{L^2}
+C\Big(1+\|\Delta\phi_1\|_{L^2}+\|\Delta\phi_1^\ast\|_{L^2}\Big)\|\tilde{\phi}_1\|_{L^2}.
\end{align*}
Similarly to the estimate of $J_1$, we get
\begin{align*}
|J_2| \leq C\big(1+\|\Delta\phi_1^\ast\|_{L^2}^{\frac12}+\|\Delta\phi_2\|_{L^2}^{\frac12}\big)\big(\|\nabla\tilde{\phi}_1\|_{L^2}+\|\nabla\tilde{\phi}_2\|_{L^2}\big) + C(1+\|\Delta\phi_1\|_{L^2}+\|\Delta\phi_2\|_{L^2})\|\tilde{\phi}_1\|_{L^2}.
\end{align*}
Besides,
\begin{align*}
|J_3|&\leq\frac{1}{\beta_1\beta_1^{\ast}}\left|\int_{\T} \Big( f'(\phi_1^\ast)|^2-|f'(\phi_1)|^2 \Big) \,\ud{x}\int_{\T}W'(\phi_1)f'(\phi_1)\,\ud{x}\right|\\
&\qquad+\frac{1}{\beta_1^{\ast}}\left| \int_{\T}\Big( W'(\phi_1)f'(\phi_1)-W'(\phi_1^\ast)f'(\phi_1^\ast) \Big)\,\ud{x} \right|\\
&\leq C\left| \int_{\T}\Big( f'(\phi_1^\ast)-f'(\phi_1) \Big) \Big( f'(\phi_1^\ast)+f'(\phi_1) \Big)\,\ud{x} \right|\\
&\qquad+C \left| \int_{\T}\Big( W'(\phi_1)-W'(\phi_1^\ast) \Big) f'(\phi_1)+W'(\phi_1^\ast) \Big( f'(\phi_1)-f'(\phi_1^\ast) \Big)\,\ud{x}\right| \\
&\leq C\|\tilde{\phi}_1\|_{L^2}\|f''(\eta_1)\|_{L^6}\big\|f'(\phi_1^\ast)+f'(\phi_1)\big\|_{L^3}
+C\|\tilde{\phi}_1\|_{L^2}\|W''(\eta_1)\|_{L^3}\|f'(\phi_1)\big\|_{L^6}\\
&\qquad+C\|\tilde{\phi}_1\|_{L^2}\|W'(\phi_1^\ast)\|_{L^3}\|f''(\eta_1)\big\|_{L^6}
\qquad\mbox{where }\, \eta_1\in (\phi_1, \phi_1^\ast)\\
&\leq C\|\tilde{\phi}_1\|_{L^2}(1+\|\phi_1\|_{L^\infty}+\|\phi_1^\ast\|_{L^\infty})\\
&\leq C\big(1+\|\Delta\phi_1\|_{L^2}^{\frac12}+\|\Delta\phi_1^\ast\|_{L^2}^{\frac12}\big)\|\tilde{\phi}_1\|_{L^2}.
\end{align*}
Similar to $J_3$, we have
\begin{align*}
|J_4|\leq C\big(1+\|\Delta\phi_1\|_{L^2}^{\frac12}+\|\Delta\phi_1^\ast\|_{L^2}^{\frac12}+\|\Delta\phi_2^\ast\|_{L^2}^{\frac12}\big)
\big(\|\tilde{\phi}_1\|_{L^2}+\|\tilde{\phi}_2\|_{L^2}\big).
\end{align*}
Meanwhile, it is easy to check from \eqref{HLS-1} that
\begin{align*}
|J_5+J_6|&\leq\frac{1}{\beta_1\beta_1^{\ast}}\int_{\T}\big|f'(\phi_1)-f'(\phi_1^\ast)\big|\big|f'(\phi_1)+f'(\phi_1^\ast)\big|\,\ud{x}
\sum_{k=1}^2\int_{\T}|\gamma_{1k}\big|(-\Delta)^{-1}(f(\phi_k)-\omega_k)\big||f'(\phi_1)|\,\ud{x}\\
&\qquad+\frac{1}{\beta_1^\ast}\sum_{k=1}^2\int_{\T}\big|\gamma_{1k}\big|\big|(-\Delta)^{-1}(f(\phi_k)-f(\phi_k^\ast))\big|\big|f'(\phi_1)\big|\,\ud{x}\\
&\qquad+\frac{1}{\beta_1^\ast}\sum_{k=1}^2\int_{\T}\big|\gamma_{1k}\big|\big|(-\Delta)^{-1}(f(\phi_k^\ast)-\omega_k)\big|\big|f'(\phi_1)-f'(\phi_1^\ast)\big|\,\ud{x}\\
&\leq C\|\tilde{\phi}_1\|_{L^2}+C\sum_{k=1}^2\big\|f(\phi_k)-f(\phi_k^\ast)\big\|_{L^{\frac65}}\|f'(\phi_1)\|_{L^{\frac65}}+C\|\tilde{\phi}_1\|_{L^2}\\
&\leq C\|\tilde{\phi}_1\|_{L^2}+C\sum_{k=1}^2\|f'(\eta_k)\|_{L^3}\|\tilde{\phi}_k\|_{L^2}    \qquad\mbox{where }\, \eta_k\in (\phi_k, \phi_k^\ast),
\;1\leq k\leq 2\\
&\leq C\|\tilde{\phi}_1\|_{L^2}+C\|\tilde{\phi}_2\|_{L^2}.
\end{align*}
Summing up all the above estimates from $J_1$ to $J_6$, after using Young's inequality we conclude that
\begin{align}\label{estimate-difference-lagrange}
|\lambda_1(t)-\lambda_1^\ast(t)|
&\leq C\Big[1+\sum_{k=1}^2(\|\Delta\phi_k\|_{L^2}+\|\Delta\phi_k^\ast\|_{L^2})\Big]\sum_{k=1}^2\big(\|\tilde{\phi}_k\|_{L^2}+\|\nabla\tilde{\phi}_k^\ast\|_{L^2}\big),
\end{align}
while the estimate for $|\lambda_2(t)-\lambda_2^\ast(t)|$ is identical to \eqref{estimate-difference-lagrange}.

As a consequence, by \eqref{estimate-difference-lagrange} and Young's inequality, we obtain
\begin{align}\label{I-6}
I_6\leq C\Big[1+\sum_{k=1}^2(\|\Delta\phi_k\|_{L^2}^2+\|\Delta\phi_k^\ast\|_{L^2}^2)\Big]\sum_{k=1}^2\|\tilde{\phi}_k(t,\cdot)\|_{L^2}
+\frac14\sum_{k=1}^2\|\nabla\tilde{\phi}_k(t,\cdot)\|_{L^2}^2.
\end{align}

In conclusion, summing up \eqref{I-1}-\eqref{I-5} and \eqref{I-6}, we arrive at the inequality
\begin{align}\label{energy-inequality}
\frac{\ud}{\ud{t}} \sum_{k=1}^2\|\tilde{\phi}_k(t,\cdot)\|_{L^2} \leq C\Big[1+\sum_{k=1}^2(|\lambda_k(t)|^2+\|\Delta\phi_k(t,\cdot)\|_{L^2}^2+\|\Delta\phi_k^\ast(t,\cdot)\|_{L^2}^2)\Big]\sum_{k=1}^2\|\tilde{\phi}_k(t,\cdot)\|_{L^2}.
\end{align}
Note that $\lambda_k(t), \Delta\phi_k(t,\cdot), \Delta\phi_k^\ast(t,\cdot)\in L^2(0, T; L^2(\T))$ for $1\leq k\leq 2$, hence a direct application of Gronwall's inequality to \eqref{IC-difference-equ} yields
$
\tilde{\phi}_1(t,\cdot)=\tilde{\phi}_2(t,\cdot)\equiv 0,
$
which finishes the proof.
\end{proof}

\begin{remark}
Theorem \ref{theorem-main} is still valid if  $\T$ is replaced by any smooth and bounded domain in $\RR^3$, provided that homogeneous Dirichlet boundary conditions are imposed. Besides Theorem \ref{theorem-main} is also valid for two-dimensional case.
\end{remark}

\begin{remark}
While Theorem \ref{theorem-main} is in regard to the wellposedness of the ACON system in Lagrange multiplier form, a direct application of the De Giorgi's minimization movement scheme can also lead to the wellposedness of the ACON system in penalty form as follows:
\begin{align}
&\partial_t\phi_i=\Delta\phi_i+\frac{\Delta\phi_j}{2}-\frac12\frac{\partial{W}_{\mathrm{T}}}{\partial\phi_i}
-\sum_{k=1}^2\gamma_{ik}(-\Delta)^{-1}\big(f(\phi_k)-\omega_k\big)f'(\phi_i) - M \int_{\T} (f(\phi_i)-\omega_i)\text{d}x \cdot f'(\phi_i), \\
&\phi_i(x,0)=\phi_{i0}(x),\quad i=1,2,
\end{align}
where $M\gg1$ is the penalty constant. Indeed, in the penalty form, one does not need to handle any singularity arising from nontrivial denominators, which makes the application of the De Giorgi's minimization movement scheme much more straightforward.
\end{remark}

\begin{remark}
The wellposedness of the Allen-Cahn-Ohta-Kawasaki (ACOK) equation \cite{XuZhao_submitted2020}, the binary counterpart of the ACON system, either in the Lagrange multiplier form or penalty form, can be similarly established by following the De Giorgi's minimization movement scheme.
\end{remark}



\section{Concluding remarks}

In this paper, we prove the global well-posedness of the ACON system with two fixed nonlinear volume constraints. Different from the standard De Giorgi's minimizing movement scheme, we identify the limit curve first and use an approximation of this limit curve to establish the nonlinear terms caused by the nonlinear volume constraint in the discrete Euler Lagrange equation. This special treatment can be potentially use to study the well-posedness of other $L^2$ gradient flow dynamics with nonlinear constraints. 


\section{Appendix}

In the appendix, we shall derive $\forall\,n\geq N$ the
corresponding Euler-Lagrange equation for the minimizer
$\big(\phi_{1\tau_n}^{k+1}, \phi_{2\tau_n}^{k+1}\big)$ to the functional
\begin{align}
F_{\tau}[\phi_1, \phi_2; \phi_{1\tau_n}^{k}, \phi_{2\tau_n}^{k}]= F_{\tau}[\phi_1, \phi_2] + \dfrac{\|\phi_1-\phi_{1\tau_n}^{k}\|_{L^2(\T)}^2+\|\phi_2-\phi_{2\tau_n}^{k}\|_{L^2(\T)}^2}{2\tau},
\end{align}
in the admissible set $H_{\omega_1}^1\times H_{\omega_2}^1$. This is an adapted version of \cite[Theorem 2, Section 8.4]{Evans_Book1998}, but for the sake of completeness we provide all details here. In the sequel the index $i$ ranges from $1$ to $2$.

\smallskip

\noindent\textbf{Step 1.} Let $v_1, v_2\in H^1(\T)$ be two independent functions. By
\eqref{nonzero-condition}, we know that
\[
f'(\phi_{i\tau_n}^{k+1})\;\mbox{is not equal to zero a.e. within } \T, \quad i=1,2.
\]
And by the choice of \eqref{test-function}, we have
\begin{equation}\label{j-1}
\int_{\T}f'(\phi_{i\tau_n}^{k+1}(x)) w_{i\tau_n}^{k+1}(x)\,\ud{x}\neq
0,\qquad \;\forall\ n\geq N, \  \forall\,0\leq  k\leq N_n-1.
\end{equation}
Let us consider the following two functions
\begin{equation}
j_i(\delta,\sigma):=\int_{\T}\Big[f\big(\phi_{i\tau_n}^{k+1}+\delta {v_i}+\sigma w_{i\tau_n}^{k+1}\big)-\omega_i\Big]\,\ud{x}.
\end{equation}
Then it is clear that
\begin{equation}
j_i(0,0)=\int_{\T}\big[f(\phi_{i\tau_n}^{k+1})-\omega_i\big]\,\ud{x}=0.
\end{equation}
Besides, $j$ is $C^1$ and satisfies
\begin{align}
\frac{\partial{j}_i}{\partial\delta}(\delta,\sigma) & =\int_{\T}f'\big(\phi_{i\tau_n}^{k+1}+\delta {v_i}+\sigma w_{i\tau_n}^{k+1}\big)v_i(x)\,\ud{x},\label{j-2}\\
\frac{\partial{j}_i}{\partial\sigma}(\delta,\sigma) & =\int_{\T}f'\big(\phi_{i\tau_n}^{k+1}+\delta {v_i}+\sigma w_{i\tau_n}^{k+1}\big) w_{i\tau_n}^{k+1}(x)\,\ud{x}.
\label{j-3}
\end{align}
Note that \eqref{j-1} implies
\[
\frac{\partial{j}_i}{\partial\sigma}(0, 0)\neq 0.
\]
As a consequence, using implicit function theorem, there exist $C^1$ functions $\eta_i:\RR\rightarrow\RR$ satisfying
\begin{align}
\eta_i(0)&=0,\\
j_i(\delta,\eta_i(\delta))&=0, \quad\mbox{for all sufficiently small }
|\delta|\leq\delta_0,\label{j-4}
\end{align}
for some $\delta_0>0$. Then we obtain after differentiating both sides of \eqref{j-4} that
\[
\frac{\partial{j}_i}{\partial\delta}(\delta,\eta_i(\delta))+\frac{\partial{j}_i}{\partial\sigma}(\delta,\eta_i(\delta))\eta'_i(\delta)=0,
\]
which together with \eqref{j-2} and \eqref{j-3} gives
\begin{align}\label{j-5}
\eta'_i(0)=-\dfrac{\int_{\T}f'\big(\phi_{i\tau_n}^{k+1})v_i(x)\,\ud{x}}{\int_{\T}f'\big(\phi_{i\tau_n}^{k+1}) w_{i\tau_n}^{k+1}(x)\,\ud{x}}
\end{align}

\noindent\textbf{Step 2.} Next let us define
\[
I(\delta):=F_{\tau}\big[\phi_{1\tau_n}^{k+1}+\delta {v_1}+\eta_1(\delta) w_{1\tau_n}^{k+1},\,\phi_{2\tau_n}^{k+1} + \delta {v_2}+\eta_2(\delta) w_{2\tau_n}^{k+1}; \phi_{1\tau_n}^{k}, \phi_{1\tau_n}^{k}\big].
\]
By \eqref{j-4}, $\phi_{i\tau_n}^{k+1}+\delta {v_i}+\eta_i(\delta) w_{i\tau_n}^{k+1} \in H_{\omega_i}^1$, $\forall\,|\delta|\leq\delta_0$, $i=1,2$. Thus the $C^1$ function $I(\cdot)$ takes the minimum value at $0$, which yields $0=I'(0)$. Since $v_1$ and $v_2$ are independent, we get after expansion
\begin{align*}
0&=\int_{\T}\left[\dfrac{\phi_{1\tau_n}^{k+1}-\phi_{1\tau_n}^k}{\tau_n}+ \frac12 W'(\phi_{1\tau_n}^{k+1})- \frac12 W'(1-\phi_{1\tau_n}^{k+1}-\phi_{2\tau_n}^{k+1}) \right]\big[v_1(x)+\eta_1'(0) w_{1 n}^{k+1}(x)\big]\,\ud{x}\non\\
&\quad+\int_{\T}  \sum_{l=1}^2\gamma_{1l}(-\Delta)^{-1}\big(f(\phi_{l\tau_n}^{k+1})-\omega_l\big)
f'(\phi_{1\tau_n}^{k+1})\big[v_1(x)+\eta_1'(0) w_{1\tau_n}^k(x)\big]\,\ud{x} \non\\
&\quad+\int_{\T}\Big(\nabla\phi_{1\tau_n}^{k+1}+\frac{1}{2}\nabla\phi_{2\tau_n}^{k+1}\Big)\big[v_1(x)+\eta_1'(0) w_{1\tau_n}^{k+1}(x)\big]\,\ud{x},
\end{align*}
and
\begin{align*}
0&=\int_{\T}\left[\dfrac{\phi_{2\tau_n}^{k+1}-\phi_{2\tau_n}^k}{\tau_n}+ \frac12 W'(\phi_{2\tau_n}^{k+1})- \frac12 W'(1-\phi_{1\tau_n}^{k+1}-\phi_{2\tau_n}^{k+1}) \right]\big[v_2(x)+\eta_2'(0) w_{2 n}^{k+1}(x)\big]\,\ud{x}\non\\
&\quad+\int_{\T}  \sum_{l=1}^2\gamma_{2l}(-\Delta)^{-1}\big(f(\phi_{l\tau_n}^{k+1})-\omega_l\big)
f'(\phi_{2\tau_n}^{k+1})\big[v_2(x)+\eta_2'(0) w_{2\tau_n}^k(x)\big]\,\ud{x} \non\\
&\quad+\int_{\T}\Big(\nabla\phi_{2\tau_n}^{k+1}+\frac{1}{2}\nabla\phi_{1\tau_n}^{k+1}\Big)\big[v_2(x)+\eta_2'(0) w_{2\tau_n}^{k+1}(x)\big]\,\ud{x}.
\end{align*}
Define $\lambda_{i\tau_n}^{k+1}$ as in (\ref{lagrange-multiplier}), then the above two equations lead to the Euler-Lagrange equations (\ref{eqn:EulerLagrange}).


\section{Acknowledgements}

S. Joo would like to acknowledge support from the National Science Foundation through grant \#DMS-1909268 and Simons Foundation Grant No. 422622.
X. Xu's work is supported by a grant from the Simons Foundation through grant No. 635288. Y. Zhao's work is supported by a grant from the Simons Foundation through Grant No. 357963 and the Columbian College Facilitating Funds (CCFF) of George Washington University.


%
%
%
%
%
%
%
%

\bibliography{OhtaKawasaki}

\end{document}